%% file: Computational_paper.tex
\documentclass[a4paper,11pt,oneside]{amsart}

\usepackage[german, english]{babel}
\usepackage{graphicx}
\usepackage{microtype}
\usepackage{color}
\usepackage{xcolor}
\usepackage{amsmath}
\usepackage{amssymb}
\usepackage{amsfonts}
\usepackage{amsthm}
\usepackage{dsfont}
\usepackage{geometry}
\usepackage{media9}
\usepackage[T1]{fontenc}
\usepackage[utf8]{inputenc}
\usepackage{lmodern}
\usepackage{longtable}
\usepackage{textcomp}
\usepackage{booktabs}
\usepackage{caption}
\usepackage{subcaption}
\usepackage{setspace}
\usepackage{calc}
\usepackage{mathrsfs} 
\usepackage{tabularx}
\usepackage[arrow, matrix, curve]{xy} 
\usepackage{lscape}
\usepackage{floatflt}
\usepackage{xurl}
\usepackage[colorlinks]{hyperref}
\usepackage{siunitx}
\usepackage{cancel}
\usepackage{icomma}
\usepackage{tikz}									

\usepackage{listings}
\usetikzlibrary{shapes.geometric, arrows}	

\tikzstyle{startstop} = [minimum width=3cm, text width = 5cm, minimum height=1cm, text centered, draw=black]
\tikzstyle{arrow} = [thick, ->, >=stealth]

\usetikzlibrary{matrix}
\usetikzlibrary{patterns}
\usetikzlibrary{positioning}
\usetikzlibrary{decorations.pathmorphing}
\usetikzlibrary{cd}

\usepackage{algorithmicx}
\usepackage{algorithm}
\usepackage{algpseudocode}

\MakeRobust{\Call}

\algdef{SE}[DOWHILE]{Do}{DoWhile}{\algorithmicdo}[1]{\algorithmicwhile\ #1}
\algrenewcommand{\algorithmiccomment}[1]{\hfill $\rhd$ \emph{#1}}
\algrenewcommand{\algorithmicrequire}{\textbf{Input:}}
\algrenewcommand{\algorithmicensure}{\textbf{Output:}}
\algnewcommand{\Or}{\textbf{or}}
\algnewcommand{\And}{\textbf{and}}
\algnewcommand{\Not}{\textbf{not}\,}
\algnewcommand\algorithmicforeach{\textbf{for each}}
\algdef{S}[FOR]{ForEach}[1]{\algorithmicforeach\ #1\ \algorithmicdo}

\usepackage{xspace}

\theoremstyle{definition}
\newtheorem{definition}{Definition}[section]

\newtheorem*{que*}{Question}

\theoremstyle{remark}

\newtheorem{rem}[definition]{Remark}

\theoremstyle{plain}
\newtheorem{theorem}[definition]{Theorem}
\newtheorem*{theorem*}{Theorem}

\newtheorem{prop}[definition]{Proposition}

\newcommand\RR{{\mathbb R}}
\newcommand\ZZ{{\mathbb Z}}
\newcommand{\val}[1]{\text{val}(#1)}

\DeclareMathOperator{\Trop}{Trop}

\newcommand\BShape{B}
\newcommand\cT{{\mathcal T}}

\title{Computing tropical bitangents to smooth quartic curves in \texttt{polymake}}
\author{Alheydis Geiger}
\address[Alheydis Geiger]{Department of Mathematics, University of T\"{u}bingen, Germany}
\email{\href{mailto:alheydis.geiger@math.uni-tuebingen.de}{alheydis.geiger@math.uni-tuebingen.de}}
\author{Marta Panizzut}
\address[Marta Panizzut]{ Technische Universität Berlin, Chair of Discrete Mathematics/Geometry  }
\email{\href{mailto:panizzut@math.tu-berlin.de}{panizzut@math.tu-berlin.de}}


\begin{document}
	\bibliographystyle{plain}
	\setlength{\unitlength}{1cm}

	\begin{abstract} 
	 In this article we introduce the  recently developed \texttt{polymake} extension \texttt{TropicalQuarticCurves} and its associated database entry in \texttt{polyDB} dealing with smooth tropical quartic curves. We report on algorithms implemented to analyze  tropical  bitangents and their lifting conditions over real closed valued fields.  The new functions and data were used by the authors to provide a tropical proof of Pl\"ucker and Zeuthen's count of real bitangents to smooth quartic curves. 

	\end{abstract}
	\maketitle	
\section{Introduction}

This paper is concerned with computational aspects in the study of tropical bitangents to smooth tropical plane quartic curves. We introduce the \texttt{polymake} extension \texttt{TropicalQuarticCurves} \cite{extension:tropquartics}  and a new entry in the database framework \texttt{polyDB}~\cite{polydb:paper}.

Tropical geometry is a combinatorial shadow of algebraic geometry. It combines methods of polyhedral geometry and combinatorics to investigate problems in algebraic geometry. Algebraic varieties can be degenerated into tropical varieties via tropicalization. Elements in the fiber of the tropicalization are called lifts of the algebraic variety. In order to apply tropical methods, it is important to understand the lifting behavior. This work provides a foundational and exhaustive computational study of real lifting of bitangents to tropical quartic curves.

The counts of complex and real bitangent lines to smooth quartic curves are classical results of Pl\"ucker and Zeuthen \cite{Plue39, Zeu73}. The analogous count for tropical curves is an example of a superabundance phenomenon in tropical geometry: Baker et al. \cite{BLMPR16} showed that smooth tropical plane quartic curves have infinitely many tropical bitangents grouped into seven equivalence classes modulo continuous translations that preserve bitangency. 

When such a phenomenon appears, it is interesting to consider lifting questions, that is, to understand which tropical bitangents are tropicalizations of algebraic bitangent lines. In \cite{LeMa19}, Len and Markwig showed that, under some genericity assumptions, four representatives (counting multiplicities) in each equivalence class lift to bitangents over the field of complex Puiseux series. Therefore, they reproduced Pl\"ucker's  count of 28 complex bitangents to generic quartic curves. 

Cueto and Markwig \cite{CueMa20} considered lifts to real bitangent lines showing that each class has either zero or four. This implies that the number of real bitangents to a generic real quartic curve is divisible by four. This did not fully recover the classical result stating that they are $4$, $8$, $16$ or $28$ real bitangents. The gap was closed by \cite[Theorem 1]{1GP21}, which provides a tropical version of Pl\"ucker and Zeuthen's count.  The proof of this result relies on two key-steps: enumeration of deformation classes of tropical bitangents and the analysis of their lifting conditions over a real closed valued field. Both steps can be carried out computationally, as we will briefly illustrate now. 

The combinatorial type of a smooth tropical quartic curve is encoded in the regular unimodular triangulation of the fourth dilation of the standard $2$-dimensional simplex $4\Delta_2$ induced by its coefficients.  Cueto and Markwig \cite{CueMa20} classified the combinatorial structure
of the bitangent classes. The shape of a bitangent class of a tropical quartic is not fully determined by its combinatorial type, see  Figure \ref{fig:exampledef} and \cite[Example 2.1]{1GP21}. This motivated us to introduce  deformation
classes, collecting for each bitangent class the varying shapes that appear within the
same combinatorial type, see \cite[Definition 3.2]{1GP21}.  The conditions for admitting a lift to a bitangent over a
real closed valued field determined by Cueto and Markwig \cite{1GP21} only depend on the deformation class and not on the shapes \cite[Theorem 4.5]{1GP21}.

This allowed us to work only with deformation classes, combinatorial objects which are fixed by the combinatorial type of the curve and encoded in special triangles of the triangulation.  In this article, we provide details on how the enumeration of deformation classes in the $1278$ regular unimodular triangulations of $4\Delta_2$ and the analysis of their lifting conditions was carried out in \texttt{polymake} \cite{polymake:2000}, in order to obtain a tropical proof of Pl\"ucker and Zeuthen's theorem.

The triangulations computed by Brodsky et al. \cite{BJMS15} were available at the git repository \url{https://github.com/micjoswig/TropicalModuliData}. Our analysis enriched them with more data. Therefore, we decided to save them as the new collection \texttt{QuarticCurves} withing the tropical objects in the database \texttt{polyDB} \cite{polydb:paper}, together with the additional information encoding properties of tropical bitangent lines. They can be accessed via \texttt{polymake} or via an independent API. Moreover, they can be further studied in the extension \texttt{TropicalQuarticCurves} \cite{extension:tropquartics}, where our code is also available. 

As bitangents or more generally intersections are intensely studied by the tropical community, we expect the extension \texttt{TropicalQuarticCurves} and the entry in the \texttt{polymake} database \texttt{polyDB} to be useful for further research as they make computing examples on the topic easier and more accessible. Furthermore, they can be used to compute lifting conditions over arbitrary fields \cite[Remark 4.6]{1GP21}, and they could be extended to compute arithmetic multiplicities. Moreover, the provided data gives a starting point to study bitangents for higher degree curves from a tropical point of view. 

In recent years, computational methods are having a more prominent role in the proof of interesting results in tropical and algebraic geometry. We believe that our approach to this project rightly fits and addresses  the paramount questions and  challenges on how to make these methods and data  available and confirmable by the research community, following the FAIR data principles. We refer to the description of newly established consortium MaRDI \cite{MaRDI} for further details on this interesting topic. 
 
This paper is structured as follows. In Section \ref{sec:prelim}, we briefly recall main definitions about tropical curves and their bitangents, and results from the accompanying paper \cite{1GP21}. Section~\ref{sec:polymake} introduces the \texttt{polymake} extension   \texttt{TropicalQuarticCurves} and the associated database entry. We explain the new features and illustrate the usage providing code snippets. Section \ref{sec:trop.pluecker} revisits the main theorem of \cite{1GP21} and explains how the new extension was used in the proof providing computational details. 
Finally, Section \ref{sec:hyperplanes} explores details of the hyperplanes determining the change of shapes for the deformation classes. This provides useful data for current research in this area:
To understand enriched counts of bitangents as in \cite{LaVo19} tropically, deformation classes are not sufficient, as the arithmetic multiplicities vary depending on the bitangent shapes \cite{MPS21}.

\medskip
\textbf{Acknowledgments.} The authors thank Hannah Markwig, Michael Joswig and Angelica Cueto for interesting discussions on the topic. We are grateful to Lars Kastner and Benjamin Lorenz for their help with computations. We thank Hannah Markwig, Sam Payne and Kris Shaw for telling us about their current project and allowing us to mention it in our paper. The first author was funded by a scholarship of Cusanuswerk e.V. This work is a contribution to the SFB-TRR 195 'Symbolic Tools in Mathematics and their Application' of the German Research Foundation (DFG). 

\section{Preliminaries}\label{sec:prelim}

In this section, we introduce tropical plane curves and their structure, referring to \cite{MS15} for further details. Moreover, we give an account of the main definitions and results of the accompanying paper \cite{1GP21}, which classifies deformation classes of tropical bitangents. 

In the \emph{tropical semifield} $(\mathbb{R}\cup\{\infty\},\oplus,\odot)$ arithmetic is defined by $x\oplus y = \min\{x,y\}$ and $x\odot y = x+y$. A tropical polynomial  $f = \bigoplus_{\mathbf{u}=(u_1,u_2)} \lambda_{\mathbf{u}}\odot x^{\odot u_1}\odot y^{\odot u_2}$  is a finite tropical linear combination of tropical monomials  with coefficients in the tropical semiring.
The \emph{tropical plane curve} $\Trop(V(f))$ is defined as 
\begin{equation*}
	\Trop(V(f)) = \{(a,b)\in\RR\,\,| \text{ the minimum in } f(a,b) \text{ is achieved at least twice} \}.
\end{equation*}

Given a 
polynomial $F = \sum_{\mathbf{u}=(u_1,u_2)}a_{\mathbf{u}}x^{u_1}y^{u_2}$ defined over a valued field $(K, \text{val})$, there is a canonical way of turning it into a tropical polynomial, called 
 \emph{tropicalization}:
	\begin{align*}
		\Trop(F) &= \bigoplus_{\mathbf{u}=(u_1,u_2)} \val{a_{\mathbf{u}}}\odot x^{\odot u_1}\odot y^{\odot u_2} \\
		&=\min_{\mathbf{u}=(u_1,u_2)}\{\val{a_{\mathbf{u}}}+u_1\cdot x+u_2\cdot y \}.
	\end{align*}
Let $f$ be a tropical polynomial. The \emph{Newton polytope $N_f$} of $f$ is the lattice polytope conv$\{(u_1,u_2) \, |\, \lambda_{\mathbf{u}} \not = \infty\}$. The coefficients of $f$ induce a \emph{regular subdivision} of the (lattice points of the) Newton polytope by projecting the lower faces of $\text{conv}\{(\mathbf{u}, \lambda_{\mathbf{u}})\,|\,\lambda_{\mathbf{u}}\neq\infty \}\subset\RR^2\times \RR$  onto $N_f$.  We denote the subdivision induced by the coefficients $\lambda_{\mathbf{u}}$ by $\mathcal{S}_{\lambda_{\mathbf{u}}}$. 

Given a regular subdivision $\mathcal{S}$ of $N_f$, the set of coefficient vectors that induce the same subdivision $\mathcal{S}$ form a relatively open cone 
$$\Sigma(\mathcal{S}):= \{\mathbf{\lambda}\in \RR^{|P\cap\ZZ^2|} | \mathcal{S}_{\mathbb{\lambda}} = \mathcal{S} \},$$
 called the \emph{secondary cone} of $\mathcal{S}$. The collection of secondary cones forms the \emph{secondary fan} of $N_f$.
 
A subdivision is a \emph{triangulation} if it consists solely of simplices. A triangulation is \emph{unimodular} if every simplex has minimal lattice volume.
If the subdivision $\mathcal{S}$ is a unimodular triangulation, the cone $\Sigma(\mathcal{S})$ is full dimensional. A tropical curve is \emph{smooth} if the regular subdivision of $N_f$ defined by the coefficients $\lambda_{\textbf{u}}$ is a unimodular triangulation. 

The subdivision $\mathcal{S}_{\lambda_{\mathbf{u}}}$ induced by the coefficients $\lambda_{\mathbf{u}}$ of the tropical polynomial $f$ is dual to the tropical curve $\Trop(V(f))$ via an inclusion reversing bijection:
\begin{align*}
	\text{ Subdivision }  &\longleftrightarrow \text{ Tropical curve }\\	2-\text{dimensional cell } Q & \longleftrightarrow \text{ Vertex } v \\
	\text{ Edge } E & \longleftrightarrow \text{ Edge } e.
\end{align*}

The direction of the edges of the tropical curve is inward  orthogonal to the direction vector of the edges in the subdivision.
The positions of the vertices of the tropical curve can be computed as follows: Take the dual maximal dimensional cell of the subdivision and assume it is given by the convex hull of the lattice points $(i_0,j_0),\ldots,(i_k,j_k)$.
Solving the  system
\begin{align*}
	\lambda_{i_0,j_0} + i_0x+j_0y = ...= \lambda_{i_k,j_k} + i_kx+j_ky
\end{align*}
provides the position of the vertex of the tropical curve in $\RR^2$. For a proof of this duality and further details see \cite{MS15}.	

\begin{rem}
The tropical semifield  can also be defined as the set $\mathbb{R}\cup\{-\infty\}$ with $\oplus = \max$ and $\odot = +$. The $max$-convention for the tropical addition is used in \cite{CueMa20,1GP21}. However, please note that the implementation in \texttt{polymake} of subdivisions of point configurations and their secondary cones fits with the  $min$-convention, while for hyperplanes in the tropical application of\texttt{polymake} the user can choose between the $min$- and $max$-convention.

It is easy to switch between the two conventions. The tropicalization of a polynomial $f= \sum_{\mathbf{u}}c_{\mathbf{u}}\mathbf{x}^{\mathbf{u}}$ in the $max$-convention is defined as above, but we take as tropical coefficients  $c_{\mathbf{u}}\mapsto -\val{c_{\mathbf{u}}}$. The direction of each edge of the tropical curve is outward orthogonal to the one of its dual edge in the dual subdivision. 	
\end{rem}

We now present  definitions and main results on tropical bitangents of smooth quartic curves. 
The Newton polygon of a tropical smooth quartic curve is the 4-dilated 2-dimensional simplex $4\Delta_2$.
The permutation group $S_3$ acts on the lattice points of $4\Delta_2$ and consequently on its subdivisions. This can be seen by considering the homogenized lattice points of $4\Delta_2$ and letting $S_3$ act on them by permutation of the coordinates. 

As introduced in \cite[Definition 3.1]{BLMPR16}, a tropical line $\Lambda$ is bitangent to a smooth tropical plane quartic curve $\Gamma$ if their intersection $\Lambda \cap \Gamma$ has two components with stable multiplicity $2$, or one component with stable multiplicity $4$. As shown in the aforementioned paper,  a smooth tropical quartic curve has exactly $7$ equivalence classes of bitangent lines. Two bitangents are equivalent if they can be continuously translated into each other while preserving the bitangency.  

The \emph{tropical bitangent classes} to a smooth tropical quartic curve $\Gamma$ are the connected components in $\RR^2$ consisting of all the vertices of tropical bitangents in the same equivalence class. Their shapes up to $S_3$-symmetries and their respective intersection behavior with the quartic curve have been classified in \cite{CueMa20}.
These $41$ shapes are depicted in Figure~\ref{fig:Fig6}. Please note that Cueto and Markwig worked with the $max$-convention. The classification in the $min$-convention is obtained by mirroring the figures at the $(x=y)-$axis.
\begin{figure}[h]
	\includegraphics[width=10cm]{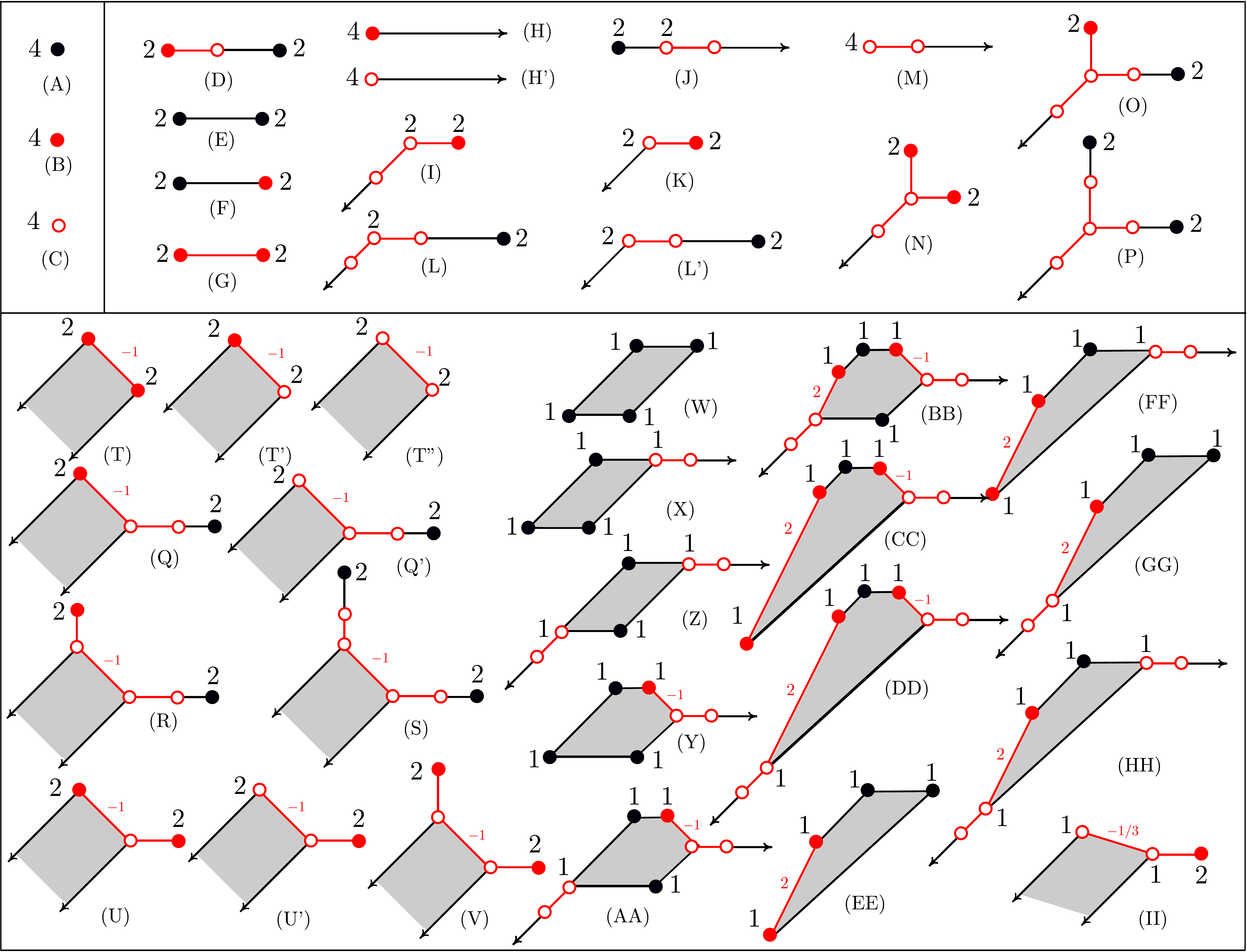}
	\caption{Shapes of bitangent classes on smooth quartics. The black numbers above the vertices indicate the lifting representatives in each class and their lifting multiplicities. Figure taken from \cite[Figure 6]{CueMa20}.}\label{fig:Fig6}
\end{figure}

It follows from the duality explained above that the shapes of the bitangent classes of a given tropical quartic impose the existence of certain subcomplexes in the dual subdivision.

\begin{definition}\cite[Definition 3.1]{1GP21} \label{def:motif}
	Let $\Gamma$ be a tropical smooth quartic curve with dual triangulation $\cT$, and let $\BShape$ be a bitangent class of $\Gamma$ of a fixed shape. The \emph{dual bitangent motif} of $(\Gamma, \BShape)$ is the subcomplex of $\cT$ that is fully determined by the shape of $\BShape$.  Dual bitangent motifs  are classified in  \cite[Figure 8]{CueMa20}.
\end{definition}

The existence of a dual bitangent motif to a certain shape in the dual subdivision $\cT$ does not imply that this shape will appear for a curve given by an arbitrary coefficient vector in $\Sigma(\cT)$.
Nevertheless, it is always possible to find a coefficient vector in $\Sigma(\cT)$ such that the shape will appear.
More precisely, given a dual subdivision, the shapes of the bitangent classes appearing for a curve $\Gamma$ of this combinatorial type can deform when continuously changing the edge lengths of $\Gamma$, see \cite[Example 2.1]{1GP21}.
This motivated the definition of deformation classes of tropical bitangents and their dual motifs. 
	
\begin{definition} \cite[Definition 3.2]{1GP21} \label{def:defclass}
	Given a tropical quartic $\Gamma_{c}$ with dual triangulation $\cT$, $c\in\Sigma(\cT)$, and a tropical bitangent class $B$, we say that a tropical bitangent class $\BShape'$ is in the same \emph{deformation class} as $\BShape$ if the following conditions are satisfied:
	\begin{itemize}
		\item[$\triangleright$] There exists $\Gamma_{c'}$ with $c' \in \Sigma(\cT)$ having $B'$ as one of its bitangent classes.
		\item[$\triangleright$] There is a continuous deformation from $\Gamma_{c}$ to $\Gamma_{c'}$ given by a path in the secondary cone $\Sigma(\cT)$ from $c$ to $c'$ that induces $\BShape$ to change to $\BShape'$.
	\end{itemize} 
	Given a unimodular triangulation $\cT$ of $4\Delta_2$ and a dual quartic curve $\Gamma$, let $\mathcal{D}$ be the deformation class of one of its seven bitangent classes. The \emph{dual deformation motif of $(\cT,\mathcal{D})$} is the union of the dual bitangent motifs of all shapes belonging to bitangent classes in  $\mathcal{D}$. Figure \ref{fig:defclasses} depicts the dual motifs of all the deformation classes, as classified in \cite{1GP21}.
\end{definition}

\begin{figure}
	\centering
	\includegraphics[width=1\textwidth]{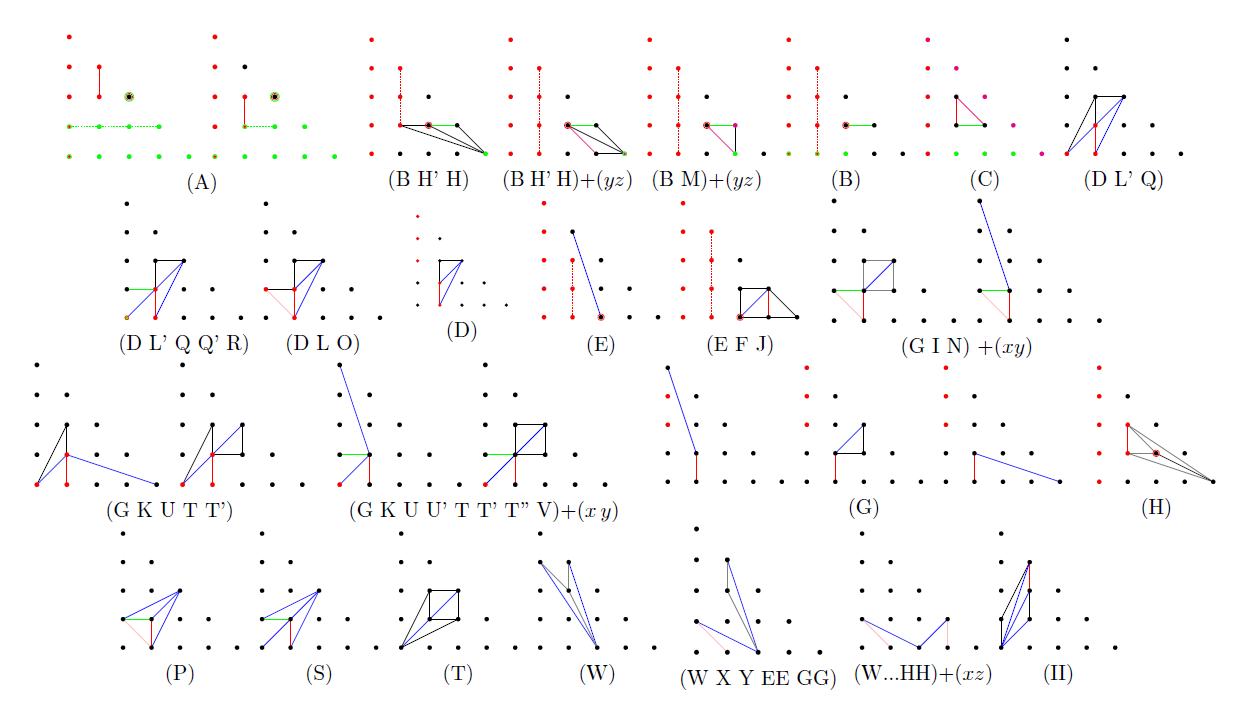}
	\caption{A list of all deformation classes of bitangent shapes.}\label{fig:defclasses}
\end{figure}

Cueto and Markwig computed conditions for an $S_3$-representative of every bitangent shape that determine whether the bitangent class of this shape would lift to real bitangents, \cite[Table 1]{CueMa20}.
These real lifting conditions are sign conditions on the coefficients of a real lift of the tropical quartic curve. 
By \cite[Theorem 4.5]{1GP21}, we know that the real lifting conditions are the same for all shapes inside the same deformation class.
Thus, the relevant signs for the lifting conditions depend only on the deformation class and the position of its dual motif inside the subdivision. 
In \cite[Table 5]{1GP21}, the real lifting conditions are presented for each deformation class.

\section{Smooth tropical quartic curves in \texttt{polymake}}\label{sec:polymake}

In this section, we present the extension \texttt{TropicalQuarticCurves} \cite{extension:tropquartics} and the associated database for combinatorial types of smooth tropical quartic curves and their deformation classes of tropical bitangents. The extension is available for \texttt{polymake} version 4.5 at 
\[
\hbox{\url{https://polymake.org/doku.php/extensions/tropicalquarticcurves}}
\]
 
 \noindent while the database entries are available in the collection \texttt{QuarticCurves} within the database \texttt{Tropical} of \texttt{polyDB} at  
\[
  \hbox{\url{https://db.polymake.org/}}
\]
The database can be accessed both via web interface and directly from \texttt{polymake}.
 \smallskip 
 
\emph{Using the extension.} In the extension \texttt{TropicalQuarticCurves}, we introduce the objects \verb|DualSubdivisionOfQuartic| and \verb|DeformationMotif| in the application fan, and the object \texttt{QuarticCurve<Addition>} in the application tropical. This procedure was inspired by the extension \texttt{TropicalCubics} \cite{JPS21} for the analysis of tropical cubic surfaces.

A \verb|DualSubdivisionOfQuartic| is a triangulation of the lattice points of $4\Delta_2$ in the fixed order illustrated in Figure \ref{fig:orderlatticepoints}, given by specifying either the weights or the maximal cells.
\begin{lstlisting}
fan > $S = new DualSubdivisionOfQuartic(MAXIMAL_CELLS=>[[0,1,2],[1,2,4], [2,4,12],[4,7,12],[2,8,12],[2,8,13],[8,12,13],[2,5,13],[5,9,13], [9,13,14],[7,11,12],[7,10,11],[4,7,10],[4,6,10],[3,4,6],[1,3,4]]);	
\end{lstlisting} 

\begin{figure}
\centering
\begin{subfigure}{0.45\textwidth}
	\centering
	\includegraphics[width=0.7\textwidth]{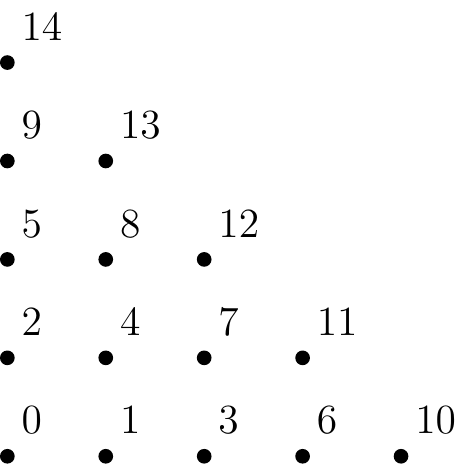}
	\caption{Order of lattice points of $4\Delta_2$ in the extension \texttt{TropicalQuarticCurves}. }\label{fig:orderlatticepoints} 
\end{subfigure}\quad
	\begin{subfigure}{0.45\textwidth}
		\centering
		\input{figures/DualSubdivisionOfQuartic.tikz}
		\caption{The dual subdivision \texttt{\$S}, with triangles of the dual deformation motif \texttt{\$DLO} filled.}\label{fig:exampleforcode}
	\end{subfigure}
\caption{}
\end{figure}
The object \verb|DeformationMotif| has three defining properties:  \verb|TRIANGLES|, \verb|TYPE| and \verb|SYMMETRY|. The property \verb|TRIANGLES| consists of the maximal cells of the triangulation that form the dual deformation motif. \verb|TYPE| is the name associated in the classification depicted in Figure \ref{fig:defclasses}. And \verb|SYMMETRY| specifies which element of the permutation group $S_3$ acts on the identity position of the deformation motif to transform it to the position indicated by the triangles. The permutation group $S_3$ is predefined under the property \verb|ACTION| of both \verb|DeformationMotif|  and \verb|DualSubdivisionOfQuartic|, so that the order of the group elements is always the same.
\begin{lstlisting}	
fan > $Motif1 = new DeformationMotif(TRIANGLES=>[[1,2,4],[2,4,12], [4,7,12]], TYPE=>"DLO",SYMMETRY=>1);
fan > print $Motif1->ACTION->ALL_GROUP_ELEMENTS;
0 1 2 3 4 5 6 7 8 9 10 11 12 13 14
0 2 1 5 4 3 9 8 7 6 14 13 12 11 10
10 6 11 3 7 12 1 4 8 13 0 2 5 9 14
10 11 6 12 7 3 13 8 4 1 14 9 5 2 0
14 9 13 5 8 12 2 4 7 11 0 1 3 6 10
14 13 9 12 8 5 11 7 4 2 10 6 3 1 0
\end{lstlisting}
When defining a \verb|DeformationMotif| from its properties, we  expect that the user makes sure that input properties are consistent with the classification in \cite{1GP21}.
There are additional properties of \verb|DeformationMotif| that can be computed from the three defining ones. For each \verb|DeformationMotif|, we can ask the real lifting conditions and the hyperplanes that describe the deformation of the shapes within the secondary cone. We refer to Section \ref{sec:hyperplanes} for more details on the hyperplanes.
\begin{lstlisting}
fan > print $Motif1->SIGN_CONDITIONS;
{-1 1 2 4 12}
{}
fan > print $Motif1->HYPERPLANES;
0 1 -1 0 1 0 0 -2 0 0 0 0 1 0 0
\end{lstlisting}

The property \verb|SIGN_CONDITIONS| returns two possibly empty sets.
The following example shows how to read the output. From \cite{CueMa20,1GP21} we know that the real lifting conditions for the deformation motif  \verb|$Motif1| is given by $-s_{01}s_{11}s_{10}s_{22}>0$. Since some deformation classes have two of these inequalities for their real lifting conditions, \verb|SIGN_CONDITIONS| is always a vector with two entries. In our example the second entry is empty, since we only have one inequality from the real lifting conditions. The lifting condition $-s_{01}s_{11}s_{10}s_{22}>0$ now translates to the set \verb|{-1 1 2 4 12}| as follows: we remember the initial minus sign by the \verb|-1| in the set. The remaining signs are assigned to coefficients of the quartic, so we encode them by the number the corresponding monomial has in the ordering fixed in the extension.

For every \verb|DualSubdivisionOfQuartic| we can compute all its $7$ dual deformation motifs and access their properties. Further, we can directly compute all the real lifting conditions associated to any generic triangulation. We refer to Section \ref{sec:trop.pluecker} for further details on the notion of genericity. 

\begin{lstlisting}
fan > print $S->IS_GENERIC;
true	
fan > $Signs = $S->ALL_SIGN_CONDITIONS;
fan > print "[".join("],[",map(join(",",@$_),@$Signs))."]\n"; #Improve output layout 
[{-1 2 8 12 13},{-1 1 2 4 12}],[{-1 2 8 12 13},{10 12}], [{-1 2 8 12 13},{10 12}],[{-1 1 2 4 12},{}],[{10 12},{}],[{10 12},{}], [{-1 1 2 4 10},{}]
\end{lstlisting}

We can also access the deformation motifs and their properties from \verb|S| as the property \verb|ALL_DEFORMATION_MOTIFS| provides an array of the deformation motifs associated to the \verb|DualSubdivisionOfQuartic|.
\begin{lstlisting}
fan > $Motifs = $S->ALL_DEFORMATION_MOTIFS;
fan > $Hyperplanes = $Motifs->[3]->HYPERPLANES;
fan > $Signconditions = $Motifs->[3]->SIGN_CONDITIONS;
fan > print $Motifs->[3]->properties;
type: DeformationMotif

ACTION
type: PermutationAction<Int, Rational>

HYPERPLANES
0 1 -1 0 1 0 0 -2 0 0 0 0 1 0 0

SIGN_CONDITIONS
{-1 1 2 4 12}
{}

SYMMETRY
1

TRIANGLES
{2 4 12}
{4 7 12}
{1 2 4}

TYPE
DLO
\end{lstlisting}

Moreover, the extension provides a function to compute the number of real bitangents for a given \verb|DualSubdivisionOfQuartic| and a given sign vector.

\begin{lstlisting}
fan > $v = new Vector<Int>([1,1,-1,1,-1,-1,1,1,1,-1,-1,1,1,1,-1]);
fan > print give_pluecker($S,$v);
4
\end{lstlisting}

In application tropical, the extension contains the new object \texttt{QuarticCurve<Addition>}, which is derived from the object \texttt{Hypersurface<Addition>}. The monomials are fixed in the order corresponding to the lattice points fixed for the \texttt{DualSubdivisionOfQuartic}, as illustrated in Figure \ref{fig:orderlatticepoints}. Thus, the \texttt{QuarticCurve} is defined by specifying the coefficients and the tropical operation \texttt{Min} or \texttt{Max}.

\begin{lstlisting}
fan > application "tropical";
tropical > $V = new Vector<Int>([-14,-9,-4,-6,0,-12,-4,0,-5,-21,-3,-1,0, -12,-31]);
tropical > $C = new QuarticCurve<Max>(COEFFICIENTS=>$V);
\end{lstlisting}

For a given curve, we can compute the shapes of the bitangent classes that appear for the chosen edge lengths.
\begin{lstlisting}
tropical > print $C->BITANGENT_SHAPES;
A A A D E E I
\end{lstlisting}

\emph{Exploring the database.} The $1278$ unimodular regular triangulations of $4\Delta_2$ computed by Brodsky et al. \cite{BJMS15} were first available at the git repository \url{https://github.com/micjoswig/TropicalModuliData}. 
They are now stored as new collection \texttt{QuarticCurves} in the database \texttt{polyDB} \cite{polydb:paper} together with additional information encoding properties of tropical bitangent lines. Each triangulation has a unique identifier, an integer between $1$ and $1278$ that can be used to retrieve it from the database.

For each triangulation given by its maximal cells, the database entry contains:
\begin{itemize}
	\item  the GKZ-vector: \verb|GKZ_VECTOR|, 
	\item a minimal representative as explained below: \verb|MINIMAL_REPRESENTATIVE|, 
	\item a boolean stating whether the triangutlation is generic: \verb|IS_GENERIC|,
	\item all dual deformation motifs: \verb|ALL_DEFORMATION_MOTIFS| and for each of these
	\begin{itemize}
		\item their sign conditions: \verb|SIGN_CONDITIONS|,
		\item their associated hyperplanes: \verb|HYPERPLANES|, 
	\end{itemize}  
	\item the Pl\"ucker number of their possible real bitangents: \verb|PLUECKER_NUMBERS|,
	\item  an exemplary sign vector for each Pl\"ucker number:  \verb|SIGN_REPRESENTATIVES|.
\end{itemize}
The \verb|GKZ_VECTOR| of a triangulation $\mathcal{T}$ is $(\text{vol}_{\cT}(p_i) \, |\, i\in \{0, \dots, 14\})$ where $p_i$ are the lattice points of $4\Delta_2$ in the fixed order and $\text{vol}_{\cT}(p_i)$ is the sum of the Euclidean volume of the simplicies in $\mathcal{T}$ containing $p_i$. The \verb|MINIMAL_REPRESENTATIVE| to a given triangulation $\mathcal{T}$ is the minimal triangulation in the $S_3$ orbit of $\mathcal{T}$ with respect to the lexicographical order on the vertex labels in the maximal cells. 

Not all these properties are displayed on the web interface. However, after downloading the JSON source file for a chosen triangulation from the database, the data can be loaded into polymake and investigated within the extension.

Given a triangulation, the property \verb|MINIMAL_REPRESENTATIVE| is also used to find the $S_3$-representative in the database. 
\begin{lstlisting}
fan > $S  = new DualSubdivisionOfQuartic(MAXIMAL_CELLS=>[[6,10,11],
[3,6,11],[3,7,11],[7,11,12],[1,3,7],[4,8,12],[8,12,13],[0,4,8],[5,9,13],
[9,13,14],[5,8,13],[2,5,8],[4,7,12],[0,4,7],[0,2,8],[0,1,7]]);
 fan > print find_in_database($S);
100
\end{lstlisting}
The function \verb|find_in_database()| computes the minimal representative of the given triangulation and compares it with the ones stored in the database returning the identifier. 

\section{Tropical enumeration of real bitangent lines}\label{sec:trop.pluecker}
We now show how the described \texttt{polymake} code was used to provide a tropical proof of Pl\"ucker and Zeuthen's count of real bitangent lines to generic plane quartic curves in \cite{1GP21}. 

\begin{theorem} \emph{\cite[Theorem 1.1]{1GP21}}  \label{thm:Plücker}
	Let $\Gamma$ be a generic tropicalization of a smooth quartic plane curve  defined over a real closed complete non-Archimedean valued field. Either  $1$, $2$, $4$ or $7$ of its bitangent classes admit a lift to real bitangents near the tropical limit. Every smooth quartic curve whose tropicalization is generic has either $4$, $8$, $16$ or $28$ totally real bitangents.
\end{theorem}

The notion of genericity used in this theorem is due to \cite{CueMa20}. A main factor is the assumption that if the tropical curve $\Gamma$ contains a vertex adjacent to three bounded edges with directions $-e_1$, $-e_2$ and $e_1+ e_2$, the shortest edge is unique. 
Up to $S_3$-symmetry there are $8$ regular unimodular triangulations of $4\Delta_2$ that can never satisfy this genericity condition. They are different completions of the subdivision shown in Figure~\ref{fig:nongeneric}.

\begin{figure}[h]
	\includegraphics[width=0.15\textwidth]{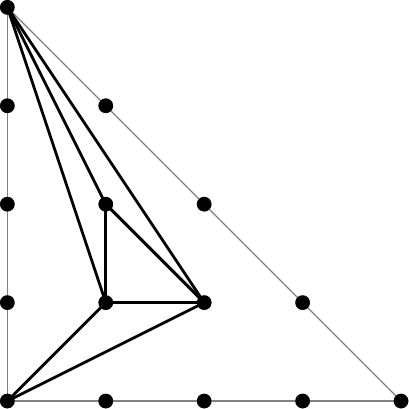}
	\caption{All tropical quartics with combinatorial type any of the $8$ refinements to a unimodular triangulation of this subdivision can never satisfy the genericity condition. }\label{fig:nongeneric}
\end{figure}
In the database entry \texttt{QuarticCurves} described in the previous section, representatives of these non-generic triangulations can be accessed under the identifiers \#511, \#718, \#841, \#904, \#1094, \#1113, \#1190 and \#1262. If the chosen triangulation is not generic, we do not return real lifting conditions for the deformation class (C), since these are not yet understood. At this point we want to recommend to use all the functions with caution, whenever you are working with a non-generic triangulation.
Furthermore, even for a generic triangulation there might be points in the secondary cone, such that the corresponding quartic curve is not generic. This set of non-generic points inside a generic secondary cone is always lower dimensional.

Let us recapitulate the idea of the proof of Theorem \ref{thm:Plücker}. First we need to enumerate the dual deformation motifs in each regular unimodular triangulation of $4\Delta_2$. 
Because of the $S_3$ action on the subdivision of $4\Delta_2$, we can restrict to representatives of the orbits under $S_3$  when investigating all unimodular triangulations of $4\Delta_2$. Such representatives were computed by Brodsky et al. \cite{BJMS15}. They are $1278$  of which $1270$ are generic.
Algorithm \ref{algo:motif-occurrences} describes how we enumerate all the dual deformation motifs for a given regular unimodular triangulation of $4\Delta_2$.

\begin{algorithm}[h] \label{alg:one}
	\caption{Finding all dual deformation motifs}
	\label{algo:motif-occurrences}
	\begin{algorithmic}[1]
		\Require{Unimodular regular triangulation $\mathcal{T}$ of $4\Delta_2$.}
		\Ensure{The list of all dual deformation motifs in $\mathcal{T}$.}
		\ForEach{dual deformation motif $M$ in Figure \ref{fig:defclasses}} 
		\ForEach{$\sigma\in S_3$}
		\If{the triangles of $\sigma(M)$ are contained in $\mathcal{T}$ }
		\State output $(\sigma(M),\sigma)$.
		\EndIf
		\EndFor
		\EndFor
	\end{algorithmic}
\end{algorithm}

Thus, we can count the number of all deformation classes for each of the 1278 representatives of the unimodular triangulations, confirming that each tropical smooth quartic curve has $7$ different associated deformation classes.

 By \cite[Theorem 4.5]{1GP21}, we know that the real lifting conditions of bitangent classes of a given tropical quartic curve depend only on the dual subdivision.
To the identity position of each deformation class we associated a matrix of two sets of integers encoding the up to two inequalities that pose the real lifting conditions.
The function detecting the dual deformation motifs  inside a given triangulation remembers the element $\sigma\in S_3$ that transforms the corresponding identity position  into the position occurring  in the triangulation. Hence, we can let $\sigma$ act on the sets of sign conditions and compute the set of real lifting conditions for any given tropical smooth quartic curve.

To count the possible number of real bitangents of a real lift of a given tropical quartic, we have to determine 
 how many of the real lifting conditions can be satisfied by a sign vector with $15$ entries.
We can standardize the sign vectors to start with $+1$, thus limiting to $2^{14}$ vectors that have to be checked in general.

\begin{algorithm}[h] \label{alg:two}
	\caption{Computing the possible numbers of real bitangents.}
	\label{algo:Plueckerscount}
	\begin{algorithmic}[1]
		\Require{Unimodular regular triangulation $\mathcal{T}$ of $4\Delta_2$.}
		\Ensure{The list of all possible numbers of real bitangents of an algebraic quartic curve with dual subdivision $\mathcal{T}$.}	
		\ForEach{$v\in\{\pm 1\}^{15}$ starting with 1}
		\State $n$ = 0
		\ForEach{lifting condition $c$ of $\mathcal{T}$}
		\If{$c(v)$ is true}
		\State $n$ + 4.
		\EndIf
		\EndFor
		\If{the value of $n$ did not appear before}
		\State output $n$
		\EndIf
		\EndFor
	\end{algorithmic}
\end{algorithm}

We implemented Algorithm 1 and 2 in \texttt{polymake} using the objects and functions introduced in the extension \texttt{TropicalQuarticCurves} obtaining a tropical computational proof that every smooth quartic curve whose tropicalization is generic has either $4$, $8$, $16$ or $28$ real bitangent lines. The code is available in the extension.

\section{Hyperplane arrangements}\label{sec:hyperplanes}
Despite the fact that the real lifting conditions only depend on the deformation class and not on the actual shape which a bitangent class has for a given tropical smooth quartic curve, it can be interesting to know which shape is realized for fixed edge lengths. 

This is important for the study of bitangent lines to quartic curves in arithmetic geometry.
For example, the arithmetic multiplicity of bitangent classes of quartic curves over any field, as investigated in \cite{MPS21}, can vary between different shapes in the same deformation class. More precisely, \cite{MPS21} provides examples of shapes with different arithmetic multiplicities that belong to the same deformation class.  One of them it is given by the shapes (W) and (BB) of deformation class(W-...-HH)$+(xz)$, which have different arithmetic multiplicities.

The deformations of the shapes inside a dual deformation class are described by linear equations, which implies that for their investigation we  consider hyperplane arrangements over a cone. 

\begin{prop}
	For a given unimodular triangulation $\cT$ of $4\Delta_2$ we can subdivide the associated secondary cone $\Sigma(\cT)$ by hyperplanes, such that for each chamber the bitangent shapes of the corresponding quartic curves are constant.
\end{prop}
\begin{proof}
 We use the notation $(\lambda_{ij})$ for a coefficient vector in $\Sigma(\cT)$ whose coordinates correspond to the lattice points $p_{ij}$ of $4\Delta_2$.  The main argument of the proof is that the deformations between shapes inside a deformation class are described by linear inequalities that provide  hyperplanes refining the secondary cone. 
	
	As introduced in Definition \ref{def:defclass}, given a tropical bitangent class, its  deformation class is  the collection of bitangent classes to which it can deforms when continuously moving  the coefficients in a path  in $\Sigma(\cT)$. In other words, a deformation of shapes happens when edge length changes in the tropical curve modify the way how the bitangent line can be  continuously moved while keeping  bitangency to the quartic.  The changes of shapes corresponding  to choosing different coefficients $(\lambda_{ij})\in\Sigma(\cT)$ are described by relative positions of   vertices and edges of the quartic curve. Therefore, they are given by linear inequalities in $(\lambda_{ij})$. See Figure \ref{fig:exampledef} for an example. 
	
 We computed the linear conditions for the deformations between shapes of each deformation class in identity position.  Table \ref{tab:all} states all the linear inequalities describing the different shapes.
	
	\begin{figure}

\begin{minipage}[c]{0.4 \textwidth}
		\centering
		\includegraphics[width=0.5\textwidth]{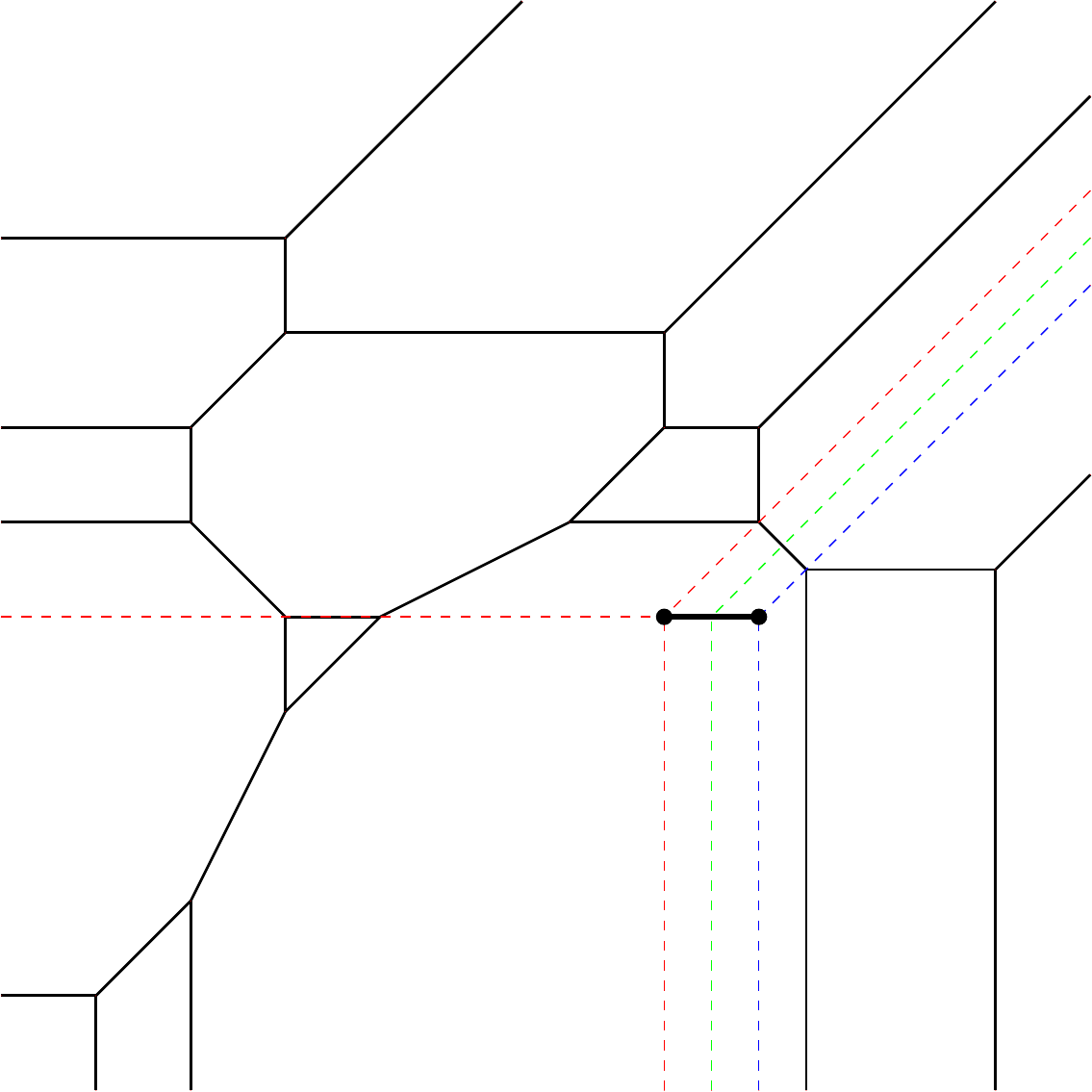}
\end{minipage}
\begin{minipage}[c]{0.5\textwidth}
	For $\lambda_1 =(0,5,5,9,8,5,6.5,9,9,4,2,7,8,7,1)$ we have $(\lambda_{1})_{8}-(\lambda_{1})_{4} < (\lambda_{1})_{11}-(\lambda_{1})_{6}$, so we see Shape (E).
\end{minipage}

\begin{minipage}[c]{0.4\textwidth}
	\centering
	\includegraphics[width=0.5\textwidth]{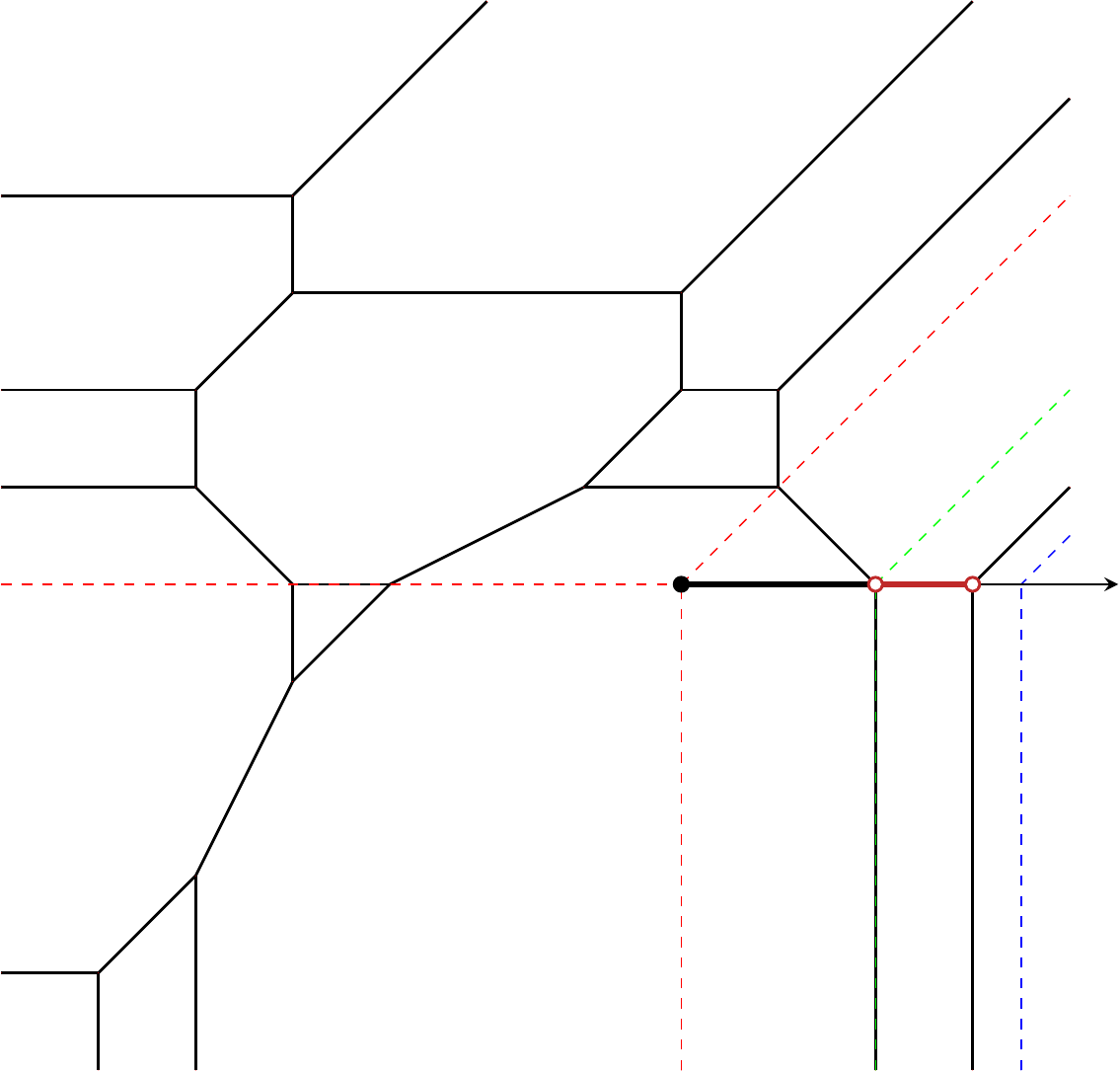}
\end{minipage}
\begin{minipage}[c]{0.5\textwidth}
		For $\lambda_2 =(0,5,5,9,8,5,6,9,9,4,2,7,8,7,1)$ 	we have  $(\lambda_{2})_{8}-(\lambda_{2})_{4} = (\lambda_{2})_{11}-(\lambda_{2})_{6}$, so we see Shape (J).
\end{minipage}

\begin{minipage}[c]{.4\textwidth}
	\centering
	\includegraphics[width=0.5\textwidth]{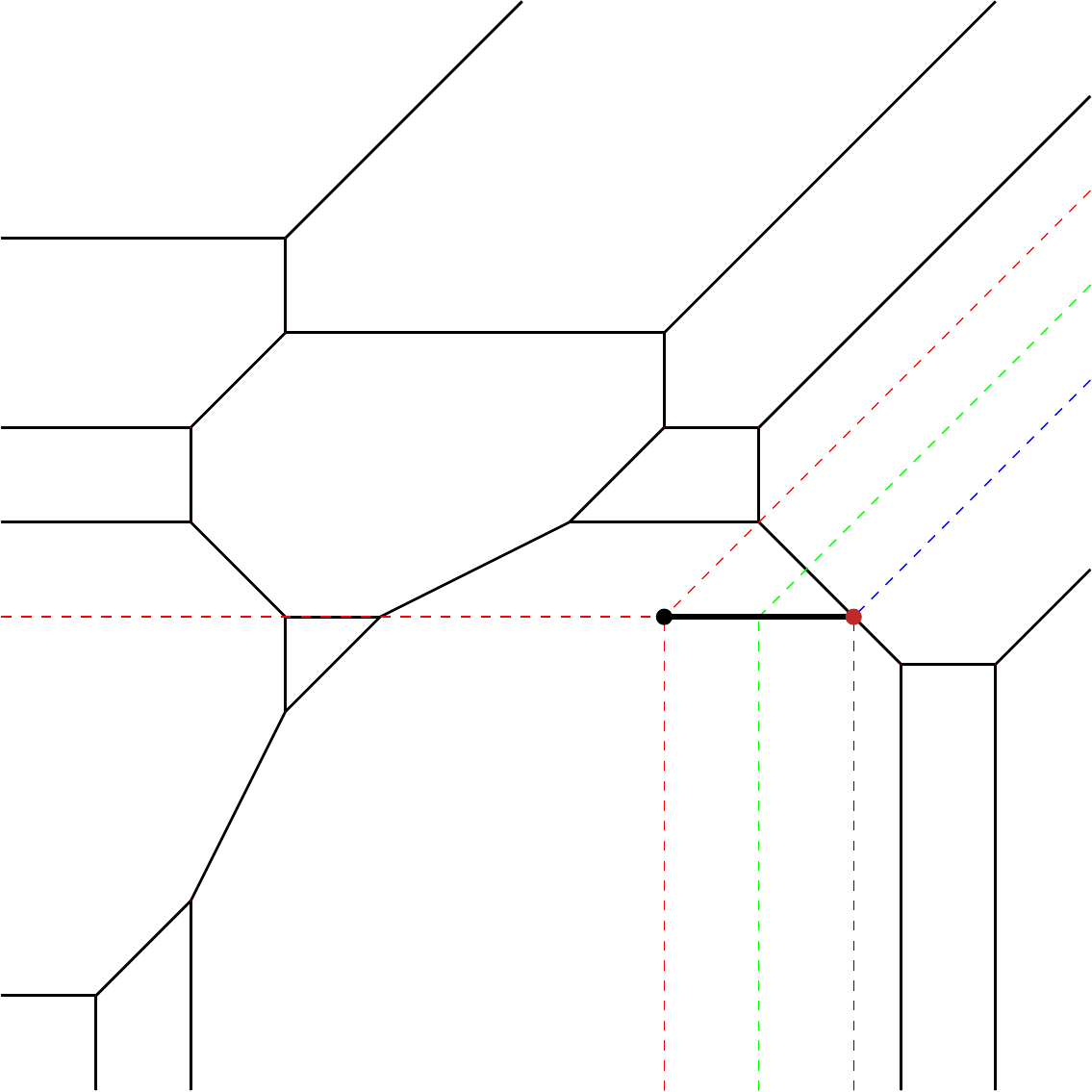}
\end{minipage}
\begin{minipage}[c]{.5\textwidth}
	For $\lambda_3 =(0,5,5,9,8,5,5.5,9,9,4,1,7,8,7,1)$ we have $ (\lambda_{3})_{8}-(\lambda_{3})_{4} > (\lambda_{3})_{11}-(\lambda_{3})_{6} $ so we see Shape (F).
\end{minipage}	
	
		\caption{The different shapes that appear for one bitangent class when choosing different edge lengths are given by linear inequalities. This example together with the figures are from \cite[Example 2.1 Figure 3]{1GP21}.}\label{fig:exampledef}
\end{figure}

\end{proof}

The hyperplanes that determine the changes of the shapes for a deformation class are stored in \texttt{polymake} as a property of the deformation motif. We can further analyze their chamber decomposition using \texttt{polymake} object \texttt{HyperplaneArrangements}~\cite{KP20}.
\begin{lstlisting}
fan > $DS = new DualSubdivisionOfQuartic(MAXIMAL_CELLS=>[[0,1,2],[1,2,3], [2,3,4],[5,8,9],[9,13,14],[8,11,12],[3,7,8],[6,10,11],[3,6,11], [3,7,11],[8,12,13],[3,4,8],[2,5,8],[8,9,13],[7,8,11],[2,4,8]]);
fan > $DM = $DS->ALL_DEFORMATION_MOTIFS;
fan > print $DM->[6]->TYPE;
W...HH+(xz)
\end{lstlisting}
The triangulation has a dual deformation motif of type (W...HH)$+(xz)$. We now look at the chamber decomposition of the secondary cone defined by the hyperplane arrangement given by the hyperplanes associated to the deformation class.
\begin{lstlisting}
fan > $Hyps = $DM->[6]->HYPERPLANES;
fan > $SC = $DS->SECONDARY_CONE;
fan > $HA = new HyperplaneArrangement(HYPERPLANES=>$Hyps); 
fan > $CD = $HA->CHAMBER_DECOMPOSITION;
\end{lstlisting}
We now look at the first chamber in the arrangement. We extract the corresponding cone and pick a point in its relative interior. 
\begin{lstlisting}
fan > $R = $CD->RAYS;
fan > $NR = $R->minor($CD->MAXIMAL_CONES->[0], All);
fan > $C = new Cone(INPUT_RAYS=>$NR, INPUT_LINEALITY=>$HA->LINEALITY_SPACE); 
fan > $Cone = intersection($C, $SC); #intersect with the secondary cone
fan > print $Cone->DIM;
15
fan > $NP = $Cone->REL_INT_POINT; 
\end{lstlisting}
We compute the tropical quartic curve defined by the coordinates of the point we have just computed and check its bitangent shapes. 
\begin{lstlisting}
fan > application 'tropical';
fan > $H = new QuarticCurve<Min>(COEFFICIENTS=>$NP);
fan > print $H->BITANGENT_SHAPES;
A B B E F N EE
\end{lstlisting}
In this example, we see that tropical quartic curves with coefficients in the cone corresponding to the first chamber in the hyperplane arrangement have a bitangent class of shape (EE) within the deformation class (W...HH)$+(xz)$.

\begin{landscape}
	\begin{longtable}{lcl}
		deformation class	& shape &condition  \\
		\hline
		\endhead
		(B)-(H/H') & (B) & $ \lambda_{1,v+1}-\lambda_{1,v} > 2\lambda_{3 1}-\lambda_{4 0}-\lambda_{2 1} $\\
		&
		(H')& $\lambda_{1,v+1}-\lambda_{1,v}= 2\lambda_{3 1}-\lambda_{4 0}-\lambda_{2 1}$\\ 
		&(H)& $\lambda_{1,v+1}-\lambda_{1,v} < 2\lambda_{3 1}-\lambda_{4 0}-\lambda_{2 1} $\\
		\hline
		(B)-(H')-(H)-(H')$_{(yz)}$-(B)$_{(yz)}$&
		(B) 	&$	\lambda_{1,v+1}-\lambda_{1,v} > 2\lambda_{3 1}-\lambda_{4 0}-\lambda_{2 1} $\\
		&(H') 	&$	\lambda_{1,v+1}-\lambda_{1,v} = 2\lambda_{3 1}-\lambda_{4 0}-\lambda_{2 1} $\\
		&(H) &$		\lambda_{1,v+1}-\lambda_{1,v} < 2\lambda_{3 1}-\lambda_{4 0}-\lambda_{2 1} $\\
		&(H')$_{(yz)}$&$	\lambda_{1,v+1}-\lambda_{1,v} = -2\lambda_{3 0}-\lambda_{4 0}-\lambda_{2 1} $\\
		&(B)$_{(yz)}$&$	\lambda_{1,v+1}-\lambda_{1,v} <-2\lambda_{3 0}-\lambda_{4 0}-\lambda_{2 1} $\\
		\hline
		(B)-(M)-(B)$_{(yz)}$&
		(B)  	&  $ \lambda_{1,v+1}-\lambda_{1,v} > \lambda_{3 1}-\lambda_{3 0}$\\ 
		&(M)  	  &$\lambda_{1,v+1}-\lambda_{1,v} = \lambda_{3 1}-\lambda_{3 0} $\\
		&(B)$_{(yz)}$&$ \lambda_{1,v+1}-\lambda_{1,v} < \lambda_{3 1}-\lambda_{3 0}$\\
		\hline
		(D)-(L')-(Q)&
		(D) &$ -2\lambda_{1 2}+\lambda_{1 1} +\lambda_{2 2} > 2\lambda_{1 0}-\lambda_{0 0}-\lambda_{1 1} $\\
		&(L') &$ -2\lambda_{1 2}+\lambda_{1 1} +\lambda_{2 2} = 2\lambda_{1 0}-\lambda_{0 0}-\lambda_{1 1} $\\
		&(Q)& $ 3\lambda_{1 1}-2\lambda_{1 2} +\lambda_{0 0} <-2\lambda_{1 2}+\lambda_{1 1} +\lambda_{2 2} < 2\lambda_{1 0}-\lambda_{0 0}-\lambda_{1 1} $\\
		
		\hline
		(D)-(L')-(Q)-(Q')-(R)&
		(D) &$ -2\lambda_{1 2}+\lambda_{1 1} +\lambda_{2 2}> 2\lambda_{1 0}-\lambda_{0 0}-\lambda_{1 1} $\\
		&(L') &$ -2\lambda_{1 2}+\lambda_{1 1} +\lambda_{2 2} = 2\lambda_{1 0}-\lambda_{0 0}-\lambda_{1 1} $\\
		&(Q) &$ 2\lambda_{0 1}-\lambda_{0 0} +\lambda_{1 1} <-2\lambda_{1 2}+\lambda_{1 1} +\lambda_{2 2} < 2\lambda_{1 0}-\lambda_{0 0}-\lambda_{1 1} $\\
		&(Q')&$  2\lambda_{0 1}-\lambda_{0 0} +\lambda_{1 1} =-2\lambda_{1 2}+\lambda_{1 1} +\lambda_{2 2} < 2\lambda_{1 0}-\lambda_{0 0}-\lambda_{1 1}$\\ 
		&(R)&$ -2\lambda_{1 2}+\lambda_{1 1} +\lambda_{2 2} < 2\lambda_{0 1}-\lambda_{0 0} -\lambda_{1 1} < 2\lambda_{1 0}-\lambda_{0 0}-\lambda_{1 1}$\\ 
		\hline
		(D)-(L)-(O)&
		(D) &$ -2\lambda_{1 2}+\lambda_{1 1} +\lambda_{2 2} > \lambda_{1 0}-\lambda_{0 1} $\\
		&(L) & $-2\lambda_{1 2}+\lambda_{1 1} +\lambda_{2 2} = \lambda_{1 0}-\lambda_{0 1} $\\
		&(O) &$ -2\lambda_{1 2}+\lambda_{1 1} +\lambda_{2 2} <\lambda_{1 0}-\lambda_{0 1}$\\
		\hline
		(E)-(F)-(J)&
		(E) & $\lambda_{1 v+1}-\lambda_{1 v} < \lambda_{3 1}-\lambda_{3 0}$ \\
		&(J) & $\lambda_{1 v+1}-\lambda_{1 v} = \lambda_{3 1}-\lambda_{3 0} $\\
		&(F)& $ \lambda_{1 v+1}-\lambda_{1 v} > \lambda_{3 1}-\lambda_{3 0} $\\
		\hline
		(G)-(I)-(N)-(I)$_{(x\,y)}$-(G)$_{(x\,y)}$&
		(G) & $\lambda_{10}-\lambda_{01} < -2\lambda_{0 3 }+\lambda_{11}+\lambda_{ 0 4}  < 2\lambda_{1 2 }-\lambda_{11}-\lambda_{ 0 4}$ \\
		blue edge  $\overline{(1,1),(0,4)}$&(I) & $\lambda_{10}-\lambda_{01} = -2\lambda_{0 3 }+\lambda_{11}+\lambda_{ 0 4}  < 2\lambda_{1 2 }-\lambda_{11}-\lambda_{ 0 4}$ \\
		&(N)& $ -2\lambda_{0 3 }+\lambda_{11}+\lambda_{ 0 4}  <\lambda_{10}-\lambda_{01} < 2\lambda_{1 2 }-\lambda_{11}-\lambda_{ 0 4}$ \\
		&(I)$_{(x\,y)}$ &  $ -2\lambda_{0 3 }+\lambda_{11}+\lambda_{ 0 4}  <\lambda_{10}-\lambda_{01} = 2\lambda_{1 2 }-\lambda_{11}-\lambda_{ 0 4}$ \\
		&(G)$_{(x\,y)}$ & $ -2\lambda_{0 3 }+\lambda_{11}+\lambda_{ 0 4}  < 2\lambda_{1 2 }-\lambda_{11}-\lambda_{ 0 4}<\lambda_{10}-\lambda_{01} $ \\
		\hline
		(G)-(I)-(N)-(I)$_{(x\,y)}$-(G)$_{(x\,y)}$&
		(G) & $\lambda_{10}-\lambda_{01} < -2\lambda_{1 2 }+\lambda_{11}+\lambda_{ 22}  < 2\lambda_{21}-\lambda_{11}-\lambda_{ 22}$ \\
		blue edge  $\overline{(1,1),(2,2)}$&(I) & $\lambda_{10}-\lambda_{01} = -2\lambda_{1 2 }+\lambda_{11}+\lambda_{ 22}  < 2\lambda_{21}-\lambda_{11}-\lambda_{ 22}$ \\
		&(N)& $ -2\lambda_{1 2 }+\lambda_{11}+\lambda_{ 22}  <\lambda_{10}-\lambda_{01} < 2\lambda_{21}-\lambda_{11}-\lambda_{ 22}$ \\
		&(I)$_{(x\,y)}$ &  $ -2\lambda_{1 2 }+\lambda_{11}+\lambda_{ 22}  <\lambda_{10}-\lambda_{01} = 2\lambda_{21}-\lambda_{11}-\lambda_{ 22}$ \\
		&(G)$_{(x\,y)}$ & $ -2\lambda_{1 2 }+\lambda_{11}+\lambda_{ 22}   < 2\lambda_{21}-\lambda_{11}-\lambda_{22}<\lambda_{10}-\lambda_{01}$ \\
		\hline
		(G)-(K)-(U)-(T)-(T')&
		(G) & $2\lambda_{10}-\lambda_{00} -\lambda_{11}< -2\lambda_{21}+\lambda_{11}+\lambda_{40}  < 2\lambda_{30}-\lambda_{11}-\lambda_{40}$ \\
		blue edge  $\overline{(1,1),(4,0)}$&(K) &$2\lambda_{10}-\lambda_{00} -\lambda_{11}= -2\lambda_{21}+\lambda_{11}+\lambda_{40}  < 2\lambda_{30}-\lambda_{11}-\lambda_{40}$\\
		&(U)& $-2\lambda_{21}+\lambda_{11}+\lambda_{40}  <2\lambda_{10}-\lambda_{00} -\lambda_{11}<  2\lambda_{30}-\lambda_{11}-\lambda_{40}$ \\
		&(T') &  $-2\lambda_{21}+\lambda_{11}+\lambda_{40}  <2\lambda_{10}-\lambda_{00} -\lambda_{11}=  2\lambda_{30}-\lambda_{11}-\lambda_{40}$\\
		&(T) & $-2\lambda_{21}+\lambda_{11}+\lambda_{40}<2\lambda_{30}-\lambda_{11}-\lambda_{40}   <2\lambda_{10}-\lambda_{00} -\lambda_{11} $\\
		\hline
		(G)-(K)-(U)-(T)-(T')&
		(G) & $2\lambda_{10}-\lambda_{00} -\lambda_{11}< -2\lambda_{12}+\lambda_{11}+\lambda_{22}  < 2\lambda_{21}-\lambda_{11}-\lambda_{22}$ \\
		blue edge  $\overline{(1,1),(2,2)}$&(K) &$2\lambda_{10}-\lambda_{00} -\lambda_{11}= -2\lambda_{12}+\lambda_{11}+\lambda_{22}  < 2\lambda_{21}-\lambda_{11}-\lambda_{22}$\\
		&(U)& $-2\lambda_{12}+\lambda_{11}+\lambda_{22}  <2\lambda_{10}-\lambda_{00} -\lambda_{11}<  2\lambda_{21}-\lambda_{11}-\lambda_{22}$ \\
		&(T') &  $-2\lambda_{12}+\lambda_{11}+\lambda_{22}  <2\lambda_{10}-\lambda_{00} -\lambda_{11}=  2\lambda_{21}-\lambda_{11}-\lambda_{22}$\\
		&(T) & $-2\lambda_{12}+\lambda_{11}+\lambda_{22}<2\lambda_{21}-\lambda_{11}-\lambda_{22}   <2\lambda_{10}-\lambda_{00} -\lambda_{11} $\\
		\hline
		(G)-(K)-(U)-(T)-(V)+&
		(G) & $-2\lambda_{01}+\lambda_{00} +\lambda_{11}<2\lambda_{10}-\lambda_{00} -\lambda_{11}< -2\lambda_{03}+\lambda_{11}+\lambda_{04}  < 2\lambda_{12}-\lambda_{11}-\lambda_{04}$ \\
		blue edge  $\overline{(1,1),(0,4)}$&(K) &$-2\lambda_{01}+\lambda_{00} +\lambda_{11}<2\lambda_{10}-\lambda_{00} -\lambda_{11}= -2\lambda_{03}+\lambda_{11}+\lambda_{04}  < 2\lambda_{12}-\lambda_{11}-\lambda_{04}$\\
		&(U)& $-2\lambda_{01}+\lambda_{00} +\lambda_{11}<-2\lambda_{03}+\lambda_{11}+\lambda_{04}<2\lambda_{10}-\lambda_{00} -\lambda_{11}< 2\lambda_{12}-\lambda_{11}-\lambda_{04}$ \\
		&(U') &  $-2\lambda_{01}+\lambda_{00} +\lambda_{11}=-2\lambda_{03}+\lambda_{11}+\lambda_{04}<2\lambda_{10}-\lambda_{00} -\lambda_{11}< 2\lambda_{12}-\lambda_{11}-\lambda_{04}$\\
		&(T') & $-2\lambda_{01}+\lambda_{00} +\lambda_{11}<-2\lambda_{03}+\lambda_{11}+\lambda_{04}<2\lambda_{10}-\lambda_{00} -\lambda_{11}= 2\lambda_{12}-\lambda_{11}-\lambda_{04}$\\
		&(T) & $-2\lambda_{01}+\lambda_{00} +\lambda_{11}<-2\lambda_{03}+\lambda_{11}+\lambda_{04}<2\lambda_{12}-\lambda_{11}-\lambda_{04}<2\lambda_{10}-\lambda_{00} -\lambda_{11} $\\
		&(T')$_{(x\,y)}$ & $-2\lambda_{01}+\lambda_{00} +\lambda_{11}=-2\lambda_{03}+\lambda_{11}+\lambda_{04}<2\lambda_{12}-\lambda_{11}-\lambda_{04}<2\lambda_{10}-\lambda_{00} -\lambda_{11} $\\
		&(T'')$_{(x\,y)}$ & $-2\lambda_{01}+\lambda_{00} +\lambda_{11}=-2\lambda_{03}+\lambda_{11}+\lambda_{04}<2\lambda_{12}-\lambda_{11}-\lambda_{04}=2\lambda_{10}-\lambda_{00} -\lambda_{11} $\\
		&(V) & $-2\lambda_{03}+\lambda_{11}+\lambda_{04}<-2\lambda_{01}+\lambda_{00} +\lambda_{11}<2\lambda_{10}-\lambda_{00} -\lambda_{11}<2\lambda_{12}-\lambda_{11}-\lambda_{04} $\\
		&(U')$_{(x\,y)}$ & $-2\lambda_{03}+\lambda_{11}+\lambda_{04}<-2\lambda_{01}+\lambda_{00} +\lambda_{11}<2\lambda_{10}-\lambda_{00} -\lambda_{11}=2\lambda_{12}-\lambda_{11}-\lambda_{04} $\\
		&(U)$_{(x\,y)}$ & $-2\lambda_{03}+\lambda_{11}+\lambda_{04}<-2\lambda_{01}+\lambda_{00} +\lambda_{11}<2\lambda_{12}-\lambda_{11}-\lambda_{04}<2\lambda_{10}-\lambda_{00} -\lambda_{11} $\\
		&(K)$_{(x\,y)}$ & $-2\lambda_{03}+\lambda_{11}+\lambda_{04}<-2\lambda_{01}+\lambda_{00} +\lambda_{11}=2\lambda_{12}-\lambda_{11}-\lambda_{04}<2\lambda_{10}-\lambda_{00} -\lambda_{11} $\\
		&(G)$_{(x\,y)}$ & $-2\lambda_{03}+\lambda_{11}+\lambda_{04}<2\lambda_{12}-\lambda_{11}-\lambda_{04}<-2\lambda_{01}+\lambda_{00} +\lambda_{11}<2\lambda_{10}-\lambda_{00} -\lambda_{11} $\\
		\hline
		(G)-(K)-(U)-(T)-(V)+&
		(G) & $-2\lambda_{01}+\lambda_{00} +\lambda_{11}<2\lambda_{10}-\lambda_{00} -\lambda_{11}< -2\lambda_{12}+\lambda_{11}+\lambda_{22}  < 2\lambda_{21}-\lambda_{11}-\lambda_{22}$ \\
		blue edge  $\overline{(1,1),(2,2)}$&(K) &$-2\lambda_{01}+\lambda_{00} +\lambda_{11}<2\lambda_{10}-\lambda_{00} -\lambda_{11}= -2\lambda_{12}+\lambda_{11}+\lambda_{22}  < 2\lambda_{21}-\lambda_{11}-\lambda_{22}$\\
		&(U)& $-2\lambda_{01}+\lambda_{00} +\lambda_{11}<-2\lambda_{12}+\lambda_{11}+\lambda_{22}<2\lambda_{10}-\lambda_{00} -\lambda_{11}< 2\lambda_{21}-\lambda_{11}-\lambda_{22}$ \\
		&(U') &  $-2\lambda_{01}+\lambda_{00} +\lambda_{11}=-2\lambda_{12}+\lambda_{11}+\lambda_{22}<2\lambda_{10}-\lambda_{00} -\lambda_{11}< 2\lambda_{21}-\lambda_{11}-\lambda_{22}$\\
		&(T') & $-2\lambda_{01}+\lambda_{00} +\lambda_{11}<-2\lambda_{12}+\lambda_{11}+\lambda_{22}<2\lambda_{10}-\lambda_{00} -\lambda_{11}= 2\lambda_{21}-\lambda_{11}-\lambda_{22}$\\
		&(T) & $-2\lambda_{01}+\lambda_{00} +\lambda_{11}<-2\lambda_{12}+\lambda_{11}+\lambda_{22}<2\lambda_{21}-\lambda_{11}-\lambda_{22}<2\lambda_{10}-\lambda_{00} -\lambda_{11} $\\
		&(T')$_{(x\,y)}$ & $-2\lambda_{01}+\lambda_{00} +\lambda_{11}=-2\lambda_{12}+\lambda_{11}+\lambda_{22}<2\lambda_{21}-\lambda_{11}-\lambda_{22}<2\lambda_{10}-\lambda_{00} -\lambda_{11} $\\
		&(T'')$_{(x\,y)}$ & $-2\lambda_{01}+\lambda_{00} +\lambda_{11}=-2\lambda_{12}+\lambda_{11}+\lambda_{22}<2\lambda_{21}-\lambda_{11}-\lambda_{22}=2\lambda_{10}-\lambda_{00} -\lambda_{11} $\\
		&(V) & $-2\lambda_{12}+\lambda_{11}+\lambda_{22}<-2\lambda_{01}+\lambda_{00} +\lambda_{11}<2\lambda_{10}-\lambda_{00} -\lambda_{11}<2\lambda_{21}-\lambda_{11}-\lambda_{22} $\\
		&(U')$_{(x\,y)}$ & $-2\lambda_{12}+\lambda_{11}+\lambda_{22}<-2\lambda_{01}+\lambda_{00} +\lambda_{11}<2\lambda_{10}-\lambda_{00} -\lambda_{11}=2\lambda_{21}-\lambda_{11}-\lambda_{22} $\\
		&(U)$_{(x\,y)}$ & $-2\lambda_{12}+\lambda_{11}+\lambda_{22}<-2\lambda_{01}+\lambda_{00} +\lambda_{11}<2\lambda_{21}-\lambda_{11}-\lambda_{22}<2\lambda_{10}-\lambda_{00} -\lambda_{11} $\\
		&(K)$_{(x\,y)}$ & $-2\lambda_{12}+\lambda_{11}+\lambda_{22}<-2\lambda_{01}+\lambda_{00} +\lambda_{11}=2\lambda_{21}-\lambda_{11}-\lambda_{22}<2\lambda_{10}-\lambda_{00} -\lambda_{11} $\\
		&(G)$_{(x\,y)}$ & $-2\lambda_{12}+\lambda_{11}+\lambda_{22}<2\lambda_{21}-\lambda_{11}-\lambda_{22}<-2\lambda_{01}+\lambda_{00} +\lambda_{11}<2\lambda_{10}-\lambda_{00} -\lambda_{11} $\\
		\hline
		(W)-(X)-(Y)-(EE)-(GG)
		& (W) & $\lambda_{20}-\lambda_{1 1}< \lambda_{10} -\lambda_{01}<-2\lambda{12}+\lambda_{13}+\lambda_{20}<2\lambda_{21}-\lambda_{20}-\lambda_{13}$ \\
		& (X) & $\lambda_{20}-\lambda_{1 1}< \lambda_{10} -\lambda_{01}=-2\lambda{12}+\lambda_{13}+\lambda_{20}<2\lambda_{21}-\lambda_{20}-\lambda_{13}$ \\
		& (Y) & $\lambda_{20}-\lambda_{1 1}< -2\lambda{12}+\lambda_{13}+\lambda_{20}<\lambda_{10} -\lambda_{01}<2\lambda_{21}-\lambda_{20}-\lambda_{13}$ \\
		& (EE) & $\lambda_{20}-\lambda_{1 1}< -2\lambda{12}+\lambda_{13}+\lambda_{20}<2\lambda_{21}-\lambda_{20}-\lambda_{13}<\lambda_{10} -\lambda_{01}$ \\
		& (GG) & $\lambda_{20}-\lambda_{1 1}< -2\lambda{12}+\lambda_{13}+\lambda_{20}<\lambda_{10} -\lambda_{01}=2\lambda_{21}-\lambda_{20}-\lambda_{13}$ \\
		\hline
		(W)-...-(HH)+$(xz)$ & (W) & $\lambda_{20}-\lambda_{11} < \lambda_{10}-\lambda_{01} < -2\lambda_{21}+\lambda_{31}+\lambda_{20}<2\lambda_{30}-\lambda_{20}-\lambda_{31}$\\
		& & $-2\lambda_{10}+\lambda_{01}+\lambda_{20}<2\lambda_{11}-\lambda_{20}-\lambda_{01}<\lambda_{31}-\lambda_{30}<\lambda_{21}-\lambda_{20}$\\
		& (X) & $\lambda_{20}-\lambda_{11} < \lambda_{10}-\lambda_{01} < -2\lambda_{21}+\lambda_{31}+\lambda_{20}<2\lambda_{30}-\lambda_{20}-\lambda_{31}$\\
		& & $-2\lambda_{10}+\lambda_{01}+\lambda_{20}<2\lambda_{11}-\lambda_{20}-\lambda_{01}=\lambda_{31}-\lambda_{30}<\lambda_{21}-\lambda_{20}$\\
		& (Y) & $\lambda_{20}-\lambda_{11} < \lambda_{10}-\lambda_{01} < -2\lambda_{21}+\lambda_{31}+\lambda_{20}<2\lambda_{30}-\lambda_{20}-\lambda_{31}$\\
		& & $-2\lambda_{10}+\lambda_{01}+\lambda_{20}<\lambda_{31}-\lambda_{30}<2\lambda_{11}-\lambda_{20}-\lambda_{01}<\lambda_{21}-\lambda_{20}$\\
		& (Z) & $\lambda_{20}-\lambda_{11} < \lambda_{10}-\lambda_{01} = -2\lambda_{21}+\lambda_{31}+\lambda_{20}<2\lambda_{30}-\lambda_{20}-\lambda_{31}$\\
		& & $-2\lambda_{10}+\lambda_{01}+\lambda_{20}<2\lambda_{11}-\lambda_{20}-\lambda_{01}=\lambda_{31}-\lambda_{30}<\lambda_{21}-\lambda_{20}$\\
		& (AA) & $\lambda_{20}-\lambda_{11} < \lambda_{10}-\lambda_{01} = -2\lambda_{21}+\lambda_{31}+\lambda_{20}<2\lambda_{30}-\lambda_{20}-\lambda_{31}$\\
		& & $-2\lambda_{10}+\lambda_{01}+\lambda_{20}<\lambda_{31}-\lambda_{30}<2\lambda_{11}-\lambda_{20}-\lambda_{01}<\lambda_{21}-\lambda_{20}$\\
		& (BB) & $\lambda_{20}-\lambda_{11} <  -2\lambda_{21}+\lambda_{31}+\lambda_{20}<\lambda_{10}-\lambda_{01}<2\lambda_{30}-\lambda_{20}-\lambda_{31}$\\
		& & $-2\lambda_{10}+\lambda_{01}+\lambda_{20}<\lambda_{31}-\lambda_{30}<2\lambda_{11}-\lambda_{20}-\lambda_{01}<\lambda_{21}-\lambda_{20}$\\
		& (CC) & $\lambda_{20}-\lambda_{11} <  -2\lambda_{21}+\lambda_{31}+\lambda_{20}<2\lambda_{30}-\lambda_{20}-\lambda_{31}<\lambda_{10}-\lambda_{01}$\\
		& & $-2\lambda_{10}+\lambda_{01}+\lambda_{20}<\lambda_{31}-\lambda_{30}<2\lambda_{11}-\lambda_{20}-\lambda_{01}<\lambda_{21}-\lambda_{20}$\\
		& (DD) & $\lambda_{20}-\lambda_{11} <  -2\lambda_{21}+\lambda_{31}+\lambda_{20}<\lambda_{10}-\lambda_{01}=2\lambda_{30}-\lambda_{20}-\lambda_{31}$\\
		& & $-2\lambda_{10}+\lambda_{01}+\lambda_{20}<\lambda_{31}-\lambda_{30}<2\lambda_{11}-\lambda_{20}-\lambda_{01}<\lambda_{21}-\lambda_{20}$\\
		& (EE) & $\lambda_{20}-\lambda_{11} <  -2\lambda_{21}+\lambda_{31}+\lambda_{20}<2\lambda_{30}-\lambda_{20}-\lambda_{31}<\lambda_{10}-\lambda_{01}$\\
		& & $-2\lambda_{10}+\lambda_{01}+\lambda_{20}<2\lambda_{11}-\lambda_{20}-\lambda_{01}<\lambda_{31}-\lambda_{30}<\lambda_{21}-\lambda_{20}$\\
		& (FF) & $\lambda_{20}-\lambda_{11} <  -2\lambda_{21}+\lambda_{31}+\lambda_{20}<2\lambda_{30}-\lambda_{20}-\lambda_{31}<\lambda_{10}-\lambda_{01}$\\
		& & $-2\lambda_{10}+\lambda_{01}+\lambda_{20}<2\lambda_{11}-\lambda_{20}-\lambda_{01}=\lambda_{31}-\lambda_{30}<\lambda_{21}-\lambda_{20}$\\
		& (GG) & $\lambda_{20}-\lambda_{11} <  -2\lambda_{21}+\lambda_{31}+\lambda_{20}<\lambda_{10}-\lambda_{01}=2\lambda_{30}-\lambda_{20}-\lambda_{31}$\\
		& & $-2\lambda_{10}+\lambda_{01}+\lambda_{20}<2\lambda_{11}-\lambda_{20}-\lambda_{01}<\lambda_{31}-\lambda_{30}<\lambda_{21}-\lambda_{20}$\\
		& (HH) & $\lambda_{20}-\lambda_{11} <  -2\lambda_{21}+\lambda_{31}+\lambda_{20}<\lambda_{10}-\lambda_{01}=2\lambda_{30}-\lambda_{20}-\lambda_{31}$\\
		& & $-2\lambda_{10}+\lambda_{01}+\lambda_{20}<2\lambda_{11}-\lambda_{20}-\lambda_{01}=\lambda_{31}-\lambda_{30}<\lambda_{21}-\lambda_{20}$\\
		& $+(xz)$ & to get all the $(xz)$ transposes, apply $(xz)$ to the conditions.\\
		\caption{The coefficient conditions for the shape changes of the bitangents for the identity representative to each deformation class with more than one shape.Note that $\lambda_{i,j}=\text{val}(a_{i,j})$.} \label{tab:all}
	\end{longtable}

\end{landscape}

\bibliography{bibliography}
\end{document}

%% file: figures/DualSubdivisionOfQuartic.tikz

\begin{tikzpicture}[x  = {(1cm,0cm)},
                    y  = {(0cm,1cm)},
                    z  = {(0cm,0cm)},
                    scale = 1,
                    color = {lightgray}]

  \coordinate (v0_unnamed__1) at (0, 0);
  \coordinate (v1_unnamed__1) at (1, 0);
  \coordinate (v2_unnamed__1) at (0, 1);

  \definecolor{vertexcolor_unnamed__1}{rgb}{ 0 0 0 }

  \tikzstyle{vertexstyle_unnamed__1_0} = [circle, scale=0.4, fill=vertexcolor_unnamed__1,label={[text=black, above right, align=left]:0},]
  \tikzstyle{vertexstyle_unnamed__1_1} = [circle, scale=0.4, fill=vertexcolor_unnamed__1,label={[text=black, above right, align=left]:1},]
  \tikzstyle{vertexstyle_unnamed__1_2} = [circle, scale=0.4, fill=vertexcolor_unnamed__1,label={[text=black, above right, align=left]:2},]

  \definecolor{facetcolor_unnamed__1}{rgb}{ 0.4666666667 0.9254901961 0.6196078431 }

  \definecolor{edgecolor_unnamed__1}{rgb}{ 0 0 0 }
  \tikzstyle{facetstyle_unnamed__1} = [fill=facetcolor_unnamed__1, fill opacity=0, draw=edgecolor_unnamed__1, line width=1 pt, line cap=round, line join=round]

  \draw[facetstyle_unnamed__1] (v0_unnamed__1) -- (v1_unnamed__1) -- (v2_unnamed__1) -- (v0_unnamed__1) -- cycle;

  \foreach \i in {0,1,2} {
    \node at (v\i_unnamed__1) [vertexstyle_unnamed__1_\i] {};
  }

  \coordinate (v0_unnamed__2) at (1, 0);
  \coordinate (v1_unnamed__2) at (0, 1);
  \coordinate (v2_unnamed__2) at (1, 1);

  \definecolor{vertexcolor_unnamed__2}{rgb}{ 0 0 0 }

  \tikzstyle{vertexstyle_unnamed__2_0} = [circle, scale=0.4, fill=vertexcolor_unnamed__2,label={[text=black, above right, align=left]:1},]
  \tikzstyle{vertexstyle_unnamed__2_1} = [circle, scale=0.4, fill=vertexcolor_unnamed__2,label={[text=black, above right, align=left]:2},]
  \tikzstyle{vertexstyle_unnamed__2_2} = [circle, scale=0.4, fill=vertexcolor_unnamed__2,label={[text=black, above right, align=left]:4},]

  \definecolor{facetcolor_unnamed__2}{rgb}{ 0.4666666667 0.9254901961 0.6196078431 }

  \definecolor{edgecolor_unnamed__2}{rgb}{ 0 0 0 }
  \tikzstyle{facetstyle_unnamed__2} = [fill=facetcolor_unnamed__2, fill opacity=1, draw=edgecolor_unnamed__2, line width=1 pt, line cap=round, line join=round]

  \draw[facetstyle_unnamed__2] (v0_unnamed__2) -- (v2_unnamed__2) -- (v1_unnamed__2) -- (v0_unnamed__2) -- cycle;

  \foreach \i in {0,1,2} {
    \node at (v\i_unnamed__2) [vertexstyle_unnamed__2_\i] {};
  }

  \coordinate (v0_unnamed__3) at (0, 1);
  \coordinate (v1_unnamed__3) at (1, 1);
  \coordinate (v2_unnamed__3) at (2, 2);

  \definecolor{vertexcolor_unnamed__3}{rgb}{ 0 0 0 }

  \tikzstyle{vertexstyle_unnamed__3_0} = [circle, scale=0.4, fill=vertexcolor_unnamed__3,label={[text=black, above right, align=left]:2},]
  \tikzstyle{vertexstyle_unnamed__3_1} = [circle, scale=0.4, fill=vertexcolor_unnamed__3,label={[text=black, above right, align=left]:4},]
  \tikzstyle{vertexstyle_unnamed__3_2} = [circle, scale=0.4, fill=vertexcolor_unnamed__3,label={[text=black, above right, align=left]:12},]

  \definecolor{facetcolor_unnamed__3}{rgb}{ 0.4666666667 0.9254901961 0.6196078431 }

  \definecolor{edgecolor_unnamed__3}{rgb}{ 0 0 0 }
  \tikzstyle{facetstyle_unnamed__3} = [fill=facetcolor_unnamed__3, fill opacity=1, draw=edgecolor_unnamed__3, line width=1 pt, line cap=round, line join=round]

  \draw[facetstyle_unnamed__3] (v1_unnamed__3) -- (v2_unnamed__3) -- (v0_unnamed__3) -- (v1_unnamed__3) -- cycle;

  \foreach \i in {0,1,2} {
    \node at (v\i_unnamed__3) [vertexstyle_unnamed__3_\i] {};
  }

  \coordinate (v0_unnamed__4) at (1, 1);
  \coordinate (v1_unnamed__4) at (2, 1);
  \coordinate (v2_unnamed__4) at (2, 2);

  \definecolor{vertexcolor_unnamed__4}{rgb}{ 0 0 0 }

  \tikzstyle{vertexstyle_unnamed__4_0} = [circle, scale=0.4, fill=vertexcolor_unnamed__4,label={[text=black, above right, align=left]:4},]
  \tikzstyle{vertexstyle_unnamed__4_1} = [circle, scale=0.4, fill=vertexcolor_unnamed__4,label={[text=black, above right, align=left]:7},]
  \tikzstyle{vertexstyle_unnamed__4_2} = [circle, scale=0.4, fill=vertexcolor_unnamed__4,label={[text=black, above right, align=left]:12},]

  \definecolor{facetcolor_unnamed__4}{rgb}{ 0.4666666667 0.9254901961 0.6196078431 }

  \definecolor{edgecolor_unnamed__4}{rgb}{ 0 0 0 }
  \tikzstyle{facetstyle_unnamed__4} = [fill=facetcolor_unnamed__4, fill opacity=1, draw=edgecolor_unnamed__4, line width=1 pt, line cap=round, line join=round]

  \draw[facetstyle_unnamed__4] (v1_unnamed__4) -- (v2_unnamed__4) -- (v0_unnamed__4) -- (v1_unnamed__4) -- cycle;

  \foreach \i in {0,1,2} {
    \node at (v\i_unnamed__4) [vertexstyle_unnamed__4_\i] {};
  }

  \coordinate (v0_unnamed__5) at (0, 1);
  \coordinate (v1_unnamed__5) at (1, 2);
  \coordinate (v2_unnamed__5) at (2, 2);

  \definecolor{vertexcolor_unnamed__5}{rgb}{ 0 0 0 }

  \tikzstyle{vertexstyle_unnamed__5_0} = [circle, scale=0.4, fill=vertexcolor_unnamed__5,label={[text=black, above right, align=left]:2},]
  \tikzstyle{vertexstyle_unnamed__5_1} = [circle, scale=0.4, fill=vertexcolor_unnamed__5,label={[text=black, above right, align=left]:8},]
  \tikzstyle{vertexstyle_unnamed__5_2} = [circle, scale=0.4, fill=vertexcolor_unnamed__5,label={[text=black, above right, align=left]:12},]

  \definecolor{facetcolor_unnamed__5}{rgb}{ 0.4666666667 0.9254901961 0.6196078431 }

  \definecolor{edgecolor_unnamed__5}{rgb}{ 0 0 0 }
  \tikzstyle{facetstyle_unnamed__5} = [fill=facetcolor_unnamed__5, fill opacity=0, draw=edgecolor_unnamed__5, line width=1 pt, line cap=round, line join=round]

  \draw[facetstyle_unnamed__5] (v0_unnamed__5) -- (v2_unnamed__5) -- (v1_unnamed__5) -- (v0_unnamed__5) -- cycle;

  \foreach \i in {0,1,2} {
    \node at (v\i_unnamed__5) [vertexstyle_unnamed__5_\i] {};
  }

  \coordinate (v0_unnamed__6) at (0, 1);
  \coordinate (v1_unnamed__6) at (1, 2);
  \coordinate (v2_unnamed__6) at (1, 3);

  \definecolor{vertexcolor_unnamed__6}{rgb}{ 0 0 0 }

  \tikzstyle{vertexstyle_unnamed__6_0} = [circle, scale=0.4, fill=vertexcolor_unnamed__6,label={[text=black, above right, align=left]:2},]
  \tikzstyle{vertexstyle_unnamed__6_1} = [circle, scale=0.4, fill=vertexcolor_unnamed__6,label={[text=black, above right, align=left]:8},]
  \tikzstyle{vertexstyle_unnamed__6_2} = [circle, scale=0.4, fill=vertexcolor_unnamed__6,label={[text=black, above right, align=left]:13},]

  \definecolor{facetcolor_unnamed__6}{rgb}{ 0.4666666667 0.9254901961 0.6196078431 }

  \definecolor{edgecolor_unnamed__6}{rgb}{ 0 0 0 }
  \tikzstyle{facetstyle_unnamed__6} = [fill=facetcolor_unnamed__6, fill opacity=0, draw=edgecolor_unnamed__6, line width=1 pt, line cap=round, line join=round]

  \draw[facetstyle_unnamed__6] (v1_unnamed__6) -- (v2_unnamed__6) -- (v0_unnamed__6) -- (v1_unnamed__6) -- cycle;

  \foreach \i in {0,1,2} {
    \node at (v\i_unnamed__6) [vertexstyle_unnamed__6_\i] {};
  }

  \coordinate (v0_unnamed__7) at (1, 2);
  \coordinate (v1_unnamed__7) at (2, 2);
  \coordinate (v2_unnamed__7) at (1, 3);

  \definecolor{vertexcolor_unnamed__7}{rgb}{ 0 0 0 }

  \tikzstyle{vertexstyle_unnamed__7_0} = [circle, scale=0.4, fill=vertexcolor_unnamed__7,label={[text=black, above right, align=left]:8},]
  \tikzstyle{vertexstyle_unnamed__7_1} = [circle, scale=0.4, fill=vertexcolor_unnamed__7,label={[text=black, above right, align=left]:12},]
  \tikzstyle{vertexstyle_unnamed__7_2} = [circle, scale=0.4, fill=vertexcolor_unnamed__7,label={[text=black, above right, align=left]:13},]

  \definecolor{facetcolor_unnamed__7}{rgb}{ 0.4666666667 0.9254901961 0.6196078431 }

  \definecolor{edgecolor_unnamed__7}{rgb}{ 0 0 0 }
  \tikzstyle{facetstyle_unnamed__7} = [fill=facetcolor_unnamed__7, fill opacity=0, draw=edgecolor_unnamed__7, line width=1 pt, line cap=round, line join=round]

  \draw[facetstyle_unnamed__7] (v0_unnamed__7) -- (v1_unnamed__7) -- (v2_unnamed__7) -- (v0_unnamed__7) -- cycle;

  \foreach \i in {0,1,2} {
    \node at (v\i_unnamed__7) [vertexstyle_unnamed__7_\i] {};
  }

  \coordinate (v0_unnamed__8) at (0, 1);
  \coordinate (v1_unnamed__8) at (0, 2);
  \coordinate (v2_unnamed__8) at (1, 3);

  \definecolor{vertexcolor_unnamed__8}{rgb}{ 0 0 0 }

  \tikzstyle{vertexstyle_unnamed__8_0} = [circle, scale=0.4, fill=vertexcolor_unnamed__8,label={[text=black, above right, align=left]:2},]
  \tikzstyle{vertexstyle_unnamed__8_1} = [circle, scale=0.4, fill=vertexcolor_unnamed__8,label={[text=black, above right, align=left]:5},]
  \tikzstyle{vertexstyle_unnamed__8_2} = [circle, scale=0.4, fill=vertexcolor_unnamed__8,label={[text=black, above right, align=left]:13},]

  \definecolor{facetcolor_unnamed__8}{rgb}{ 0.4666666667 0.9254901961 0.6196078431 }

  \definecolor{edgecolor_unnamed__8}{rgb}{ 0 0 0 }
  \tikzstyle{facetstyle_unnamed__8} = [fill=facetcolor_unnamed__8, fill opacity=0, draw=edgecolor_unnamed__8, line width=1 pt, line cap=round, line join=round]

  \draw[facetstyle_unnamed__8] (v0_unnamed__8) -- (v2_unnamed__8) -- (v1_unnamed__8) -- (v0_unnamed__8) -- cycle;

  \foreach \i in {0,1,2} {
    \node at (v\i_unnamed__8) [vertexstyle_unnamed__8_\i] {};
  }

  \coordinate (v0_unnamed__9) at (0, 2);
  \coordinate (v1_unnamed__9) at (0, 3);
  \coordinate (v2_unnamed__9) at (1, 3);

  \definecolor{vertexcolor_unnamed__9}{rgb}{ 0 0 0 }

  \tikzstyle{vertexstyle_unnamed__9_0} = [circle, scale=0.4, fill=vertexcolor_unnamed__9,label={[text=black, above right, align=left]:5},]
  \tikzstyle{vertexstyle_unnamed__9_1} = [circle, scale=0.4, fill=vertexcolor_unnamed__9,label={[text=black, above right, align=left]:9},]
  \tikzstyle{vertexstyle_unnamed__9_2} = [circle, scale=0.4, fill=vertexcolor_unnamed__9,label={[text=black, above right, align=left]:13},]

  \definecolor{facetcolor_unnamed__9}{rgb}{ 0.4666666667 0.9254901961 0.6196078431 }

  \definecolor{edgecolor_unnamed__9}{rgb}{ 0 0 0 }
  \tikzstyle{facetstyle_unnamed__9} = [fill=facetcolor_unnamed__9, fill opacity=0, draw=edgecolor_unnamed__9, line width=1 pt, line cap=round, line join=round]

  \draw[facetstyle_unnamed__9] (v0_unnamed__9) -- (v2_unnamed__9) -- (v1_unnamed__9) -- (v0_unnamed__9) -- cycle;

  \foreach \i in {0,1,2} {
    \node at (v\i_unnamed__9) [vertexstyle_unnamed__9_\i] {};
  }

  \coordinate (v0_unnamed__10) at (0, 3);
  \coordinate (v1_unnamed__10) at (1, 3);
  \coordinate (v2_unnamed__10) at (0, 4);

  \definecolor{vertexcolor_unnamed__10}{rgb}{ 0 0 0 }

  \tikzstyle{vertexstyle_unnamed__10_0} = [circle, scale=0.4, fill=vertexcolor_unnamed__10,label={[text=black, above right, align=left]:9},]
  \tikzstyle{vertexstyle_unnamed__10_1} = [circle, scale=0.4, fill=vertexcolor_unnamed__10,label={[text=black, above right, align=left]:13},]
  \tikzstyle{vertexstyle_unnamed__10_2} = [circle, scale=0.4, fill=vertexcolor_unnamed__10,label={[text=black, above right, align=left]:14},]

  \definecolor{facetcolor_unnamed__10}{rgb}{ 0.4666666667 0.9254901961 0.6196078431 }

  \definecolor{edgecolor_unnamed__10}{rgb}{ 0 0 0 }
  \tikzstyle{facetstyle_unnamed__10} = [fill=facetcolor_unnamed__10, fill opacity=0, draw=edgecolor_unnamed__10, line width=1 pt, line cap=round, line join=round]

  \draw[facetstyle_unnamed__10] (v0_unnamed__10) -- (v1_unnamed__10) -- (v2_unnamed__10) -- (v0_unnamed__10) -- cycle;

  \foreach \i in {0,1,2} {
    \node at (v\i_unnamed__10) [vertexstyle_unnamed__10_\i] {};
  }

  \coordinate (v0_unnamed__11) at (2, 1);
  \coordinate (v1_unnamed__11) at (3, 1);
  \coordinate (v2_unnamed__11) at (2, 2);

  \definecolor{vertexcolor_unnamed__11}{rgb}{ 0 0 0 }

  \tikzstyle{vertexstyle_unnamed__11_0} = [circle, scale=0.4, fill=vertexcolor_unnamed__11,label={[text=black, above right, align=left]:7},]
  \tikzstyle{vertexstyle_unnamed__11_1} = [circle, scale=0.4, fill=vertexcolor_unnamed__11,label={[text=black, above right, align=left]:11},]
  \tikzstyle{vertexstyle_unnamed__11_2} = [circle, scale=0.4, fill=vertexcolor_unnamed__11,label={[text=black, above right, align=left]:12},]

  \definecolor{facetcolor_unnamed__11}{rgb}{ 0.4666666667 0.9254901961 0.6196078431 }

  \definecolor{edgecolor_unnamed__11}{rgb}{ 0 0 0 }
  \tikzstyle{facetstyle_unnamed__11} = [fill=facetcolor_unnamed__11, fill opacity=0, draw=edgecolor_unnamed__11, line width=1 pt, line cap=round, line join=round]

  \draw[facetstyle_unnamed__11] (v0_unnamed__11) -- (v1_unnamed__11) -- (v2_unnamed__11) -- (v0_unnamed__11) -- cycle;

  \foreach \i in {0,1,2} {
    \node at (v\i_unnamed__11) [vertexstyle_unnamed__11_\i] {};
  }

  \coordinate (v0_unnamed__12) at (2, 1);
  \coordinate (v1_unnamed__12) at (4, 0);
  \coordinate (v2_unnamed__12) at (3, 1);

  \definecolor{vertexcolor_unnamed__12}{rgb}{ 0 0 0 }

  \tikzstyle{vertexstyle_unnamed__12_0} = [circle, scale=0.4, fill=vertexcolor_unnamed__12,label={[text=black, above right, align=left]:7},]
  \tikzstyle{vertexstyle_unnamed__12_1} = [circle, scale=0.4, fill=vertexcolor_unnamed__12,label={[text=black, above right, align=left]:10},]
  \tikzstyle{vertexstyle_unnamed__12_2} = [circle, scale=0.4, fill=vertexcolor_unnamed__12,label={[text=black, above right, align=left]:11},]

  \definecolor{facetcolor_unnamed__12}{rgb}{ 0.4666666667 0.9254901961 0.6196078431 }

  \definecolor{edgecolor_unnamed__12}{rgb}{ 0 0 0 }
  \tikzstyle{facetstyle_unnamed__12} = [fill=facetcolor_unnamed__12, fill opacity=0, draw=edgecolor_unnamed__12, line width=1 pt, line cap=round, line join=round]

  \draw[facetstyle_unnamed__12] (v0_unnamed__12) -- (v1_unnamed__12) -- (v2_unnamed__12) -- (v0_unnamed__12) -- cycle;

  \foreach \i in {0,1,2} {
    \node at (v\i_unnamed__12) [vertexstyle_unnamed__12_\i] {};
  }

  \coordinate (v0_unnamed__13) at (1, 1);
  \coordinate (v1_unnamed__13) at (2, 1);
  \coordinate (v2_unnamed__13) at (4, 0);

  \definecolor{vertexcolor_unnamed__13}{rgb}{ 0 0 0 }

  \tikzstyle{vertexstyle_unnamed__13_0} = [circle, scale=0.4, fill=vertexcolor_unnamed__13,label={[text=black, above right, align=left]:4},]
  \tikzstyle{vertexstyle_unnamed__13_1} = [circle, scale=0.4, fill=vertexcolor_unnamed__13,label={[text=black, above right, align=left]:7},]
  \tikzstyle{vertexstyle_unnamed__13_2} = [circle, scale=0.4, fill=vertexcolor_unnamed__13,label={[text=black, above right, align=left]:10},]

  \definecolor{facetcolor_unnamed__13}{rgb}{ 0.4666666667 0.9254901961 0.6196078431 }

  \definecolor{edgecolor_unnamed__13}{rgb}{ 0 0 0 }
  \tikzstyle{facetstyle_unnamed__13} = [fill=facetcolor_unnamed__13, fill opacity=0, draw=edgecolor_unnamed__13, line width=1 pt, line cap=round, line join=round]

  \draw[facetstyle_unnamed__13] (v0_unnamed__13) -- (v2_unnamed__13) -- (v1_unnamed__13) -- (v0_unnamed__13) -- cycle;

  \foreach \i in {0,1,2} {
    \node at (v\i_unnamed__13) [vertexstyle_unnamed__13_\i] {};
  }

  \coordinate (v0_unnamed__14) at (1, 1);
  \coordinate (v1_unnamed__14) at (3, 0);
  \coordinate (v2_unnamed__14) at (4, 0);

  \definecolor{vertexcolor_unnamed__14}{rgb}{ 0 0 0 }

  \tikzstyle{vertexstyle_unnamed__14_0} = [circle, scale=0.4, fill=vertexcolor_unnamed__14,label={[text=black, above right, align=left]:4},]
  \tikzstyle{vertexstyle_unnamed__14_1} = [circle, scale=0.4, fill=vertexcolor_unnamed__14,label={[text=black, above right, align=left]:6},]
  \tikzstyle{vertexstyle_unnamed__14_2} = [circle, scale=0.4, fill=vertexcolor_unnamed__14,label={[text=black, above right, align=left]:10},]

  \definecolor{facetcolor_unnamed__14}{rgb}{ 0.4666666667 0.9254901961 0.6196078431 }

  \definecolor{edgecolor_unnamed__14}{rgb}{ 0 0 0 }
  \tikzstyle{facetstyle_unnamed__14} = [fill=facetcolor_unnamed__14, fill opacity=0, draw=edgecolor_unnamed__14, line width=1 pt, line cap=round, line join=round]

  \draw[facetstyle_unnamed__14] (v1_unnamed__14) -- (v2_unnamed__14) -- (v0_unnamed__14) -- (v1_unnamed__14) -- cycle;

  \foreach \i in {0,1,2} {
    \node at (v\i_unnamed__14) [vertexstyle_unnamed__14_\i] {};
  }

  \coordinate (v0_unnamed__15) at (2, 0);
  \coordinate (v1_unnamed__15) at (1, 1);
  \coordinate (v2_unnamed__15) at (3, 0);

  \definecolor{vertexcolor_unnamed__15}{rgb}{ 0 0 0 }

  \tikzstyle{vertexstyle_unnamed__15_0} = [circle, scale=0.4, fill=vertexcolor_unnamed__15,label={[text=black, above right, align=left]:3},]
  \tikzstyle{vertexstyle_unnamed__15_1} = [circle, scale=0.4, fill=vertexcolor_unnamed__15,label={[text=black, above right, align=left]:4},]
  \tikzstyle{vertexstyle_unnamed__15_2} = [circle, scale=0.4, fill=vertexcolor_unnamed__15,label={[text=black, above right, align=left]:6},]

  \definecolor{facetcolor_unnamed__15}{rgb}{ 0.4666666667 0.9254901961 0.6196078431 }

  \definecolor{edgecolor_unnamed__15}{rgb}{ 0 0 0 }
  \tikzstyle{facetstyle_unnamed__15} = [fill=facetcolor_unnamed__15, fill opacity=0, draw=edgecolor_unnamed__15, line width=1 pt, line cap=round, line join=round]

  \draw[facetstyle_unnamed__15] (v0_unnamed__15) -- (v2_unnamed__15) -- (v1_unnamed__15) -- (v0_unnamed__15) -- cycle;

  \foreach \i in {0,1,2} {
    \node at (v\i_unnamed__15) [vertexstyle_unnamed__15_\i] {};
  }

  \coordinate (v0_unnamed__16) at (1, 0);
  \coordinate (v1_unnamed__16) at (2, 0);
  \coordinate (v2_unnamed__16) at (1, 1);

  \definecolor{vertexcolor_unnamed__16}{rgb}{ 0 0 0 }

  \tikzstyle{vertexstyle_unnamed__16_0} = [circle, scale=0.4, fill=vertexcolor_unnamed__16,label={[text=black, above right, align=left]:1},]
  \tikzstyle{vertexstyle_unnamed__16_1} = [circle, scale=0.4, fill=vertexcolor_unnamed__16,label={[text=black, above right, align=left]:3},]
  \tikzstyle{vertexstyle_unnamed__16_2} = [circle, scale=0.4, fill=vertexcolor_unnamed__16,label={[text=black, above right, align=left]:4},]

  \definecolor{facetcolor_unnamed__16}{rgb}{ 0.4666666667 0.9254901961 0.6196078431 }

  \definecolor{edgecolor_unnamed__16}{rgb}{ 0 0 0 }
  \tikzstyle{facetstyle_unnamed__16} = [fill=facetcolor_unnamed__16, fill opacity=0, draw=edgecolor_unnamed__16, line width=1 pt, line cap=round, line join=round]

  \draw[facetstyle_unnamed__16] (v0_unnamed__16) -- (v1_unnamed__16) -- (v2_unnamed__16) -- (v0_unnamed__16) -- cycle;

  \foreach \i in {0,1,2} {
    \node at (v\i_unnamed__16) [vertexstyle_unnamed__16_\i] {};
  }

\end{tikzpicture}